\documentclass{article}

\usepackage[utf8]{inputenc}
\usepackage[english]{babel}
\usepackage[a4paper,top=2cm,bottom=2cm,left=3cm,right=3cm,marginparwidth=1.75cm]{geometry}

\usepackage{enumitem}
\usepackage{array, booktabs, multirow}

\usepackage{amsmath, amssymb, mathtools}
\usepackage{amsthm}

\usepackage{algorithm, algpseudocode}

\usepackage{graphicx}
\usepackage{tikz}
\usepackage[colorlinks=true, allcolors=blue]{hyperref}
\usepackage[capitalise,noabbrev]{cleveref}

\usepackage{orcidlink}
\usepackage{authblk}

\theoremstyle{plain}
\newtheorem{theorem}{Theorem}
\newtheorem{proposition}[theorem]{Proposition}
\newtheorem{corollary}[theorem]{Corollary}
\newtheorem{lemma}[theorem]{Lemma}

\theoremstyle{definition}
\newtheorem{definition}[theorem]{Definition}
\newtheorem{example}[theorem]{Example}
\newtheorem{remark}[theorem]{Remark}

\DeclarePairedDelimiter\abs{\lvert}{\rvert}

\newcommand{\sgn}{\mathrm{sgn}}
\newcommand{\red}{\mathrm{red}}
\newcommand{\tsub}{\mathrm{TS}}
\newcommand{\sol}{\mathrm{Sol}}
\newcommand{\syl}{\mathrm{syl}}
\newcommand{\pl}{\mathrm{pl}}

\usepackage[small, disableredefinitions]{complexity}

\newclass{\PTIME}{P}
\newclass{\EDTOL}{EDT0L}

\newdimen\palinquad
\newdimen\palinaxis
\newcommand{\palinaux}[1]{\palinquad=\fontdimen6#1 2\relax \palinaxis=\fontdimen22#1 2\relax \mkern2mu\tikz[baseline=-\palinaxis]\draw[line width=0.05\palinquad] (0,0) circle (0.28\palinquad);\mkern2mu}
\newcommand{\palin}{\mathord{\mathchoice{\palinaux\textfont}{\palinaux\textfont}{\palinaux\scriptfont}{\palinaux\scriptscriptfont}}}

\title{Twisted Equations and Palindrome Length}

\author[1]{Volker Diekert\,\orcidlink{0000-0002-5994-3762}}
\author[2]{Anna~E.\ Frid}
\author[3]{Silas Natterer\,\orcidlink{0009-0004-6524-0854}}
\author[3]{Alexander Thumm\,\orcidlink{0009-0005-4240-2045}}
\affil[1]{University of Stuttgart, FMI, Stuttgart, Germany\protect\\ \texttt{diekert@fmi.uni-stuttgart.de}}
\affil[2]{Aix Marseille Université, CNRS, I2M, Marseille, France\protect\\ \texttt{anna.e.frid@gmail.com}}
\affil[3]{University of Siegen, ETI, Siegen, Germany\protect\\ \texttt{silas.natterer@uni-siegen.de}, \texttt{alexander.thumm@uni-siegen.de}}

\begin{document}
\maketitle

\begin{abstract}
Twisted equations in free groups are equations in which a finite group of symmetries acts on the free group by automorphisms and anti-automorphisms, such as inversion and reversal; moreover, the variables may be subject to rational constraints.
Algorithms from the theory of equations in virtually free groups solve such equations in \PSPACE, but they are intricate.
Our main result is an elementary algorithm for quadratic twisted equations, which extends the algorithm of Comerford and Edmunds (1981) for quadratic equations in free groups to the twisted setting and to rational constraints.
Besides deciding solvability in \PSPACE\ and yielding \EDTOL\ descriptions of solution sets, our algorithm computes, for every quadratic expression, finitely many parametric forms describing its values as freely reduced words.

As an application, we study the palindrome length of elements of free groups, that is, the least number of palindromes whose product they are.
Guided by the parametric forms that our algorithm computes for products of palindromes, we found, and then proved directly, a characterization of the freely reduced words of palindrome length at most $n$ in terms of patterns describing how palindromes are nested by twisted conjugations, thereby confirming a Catalan structure conjectured by the second author.
Consequently, these words form a context-free language, and palindrome length is computable in cubic time as well as in \LOGCFL, answering a question of Bardakov, Shpilrain, and Tolstykh (2005).

\end{abstract}

\section{Introduction}\label{sec:introduction}

Equations in free groups are among the most thoroughly studied algorithmic problems in combinatorial group theory: Makanin showed in his seminal work~\cite{Makanin82} that it is decidable whether a given equation has a solution, and Razborov~\cite{Razborov84} described the structure of their solution sets.
Later, satisfiability was shown to be decidable in \PSPACE, even in the presence of rational constraints on the values of the variables~\cite{Gutierrez00, DiekertGH05}.
For further results on equations in free groups, free monoids, and related structures, with or without constraints, see, e.g., \cite{RipsS95, Sela01, Plandowski04, KharlampovichM06, DiekertJP16, Jez16, CiobanuDE16, DiekertJKT26}.

In this article, we consider \emph{twisted equations} in free groups, in which a finite group of symmetries acts on the free group by automorphisms and anti-automorphisms, such as inversion and reversal, and in which the variables may be subject to rational constraints.
Twisted equations in which all twists are automorphisms arise naturally when solving equations in virtually free groups~\cite{DahmaniG10, DiekertE20}, whereas anti-automorphisms such as reversal appear naturally in combinatorics on words, for instance, when the values of variables are required to be palindromes.
Since anti-automorphisms can be handled via inversion, twisted equations can be solved in \PSPACE\ by results of the first author and Elder~\cite{DiekertE20} on equations in virtually free groups (\cref{pro:twisted-equations}).
While of great theoretical value, the known algorithms for equations in virtually free groups are intricate, and we are not aware of any implementation.
For \emph{quadratic} twisted equations (see \cref{def:quadratic}), we propose an elementary alternative.

Our main result (\cref{thm:quadratic,thm:parametric}) is an algorithm for quadratic twisted equations with rational constraints, that is, with the values of the variables restricted to rational subsets of the free group~$F(A)$.
Extending the algorithm of Comerford and Edmunds~\cite{ComerfordE81} for quadratic equations in free groups, it performs free reduction symbolically: basic substitutions, each replacing a variable by a word of length at most three, reveal step by step how the values of the variables cancel, while a simple weight function keeps the search space finite.

Besides deciding solvability and yielding \EDTOL\ descriptions of solution sets (\cref{cor:quadratic-equations}), our algorithm solves the associated unification problems parametrically (\cref{thm:parametric}): for every quadratic expression $w$, it computes finitely many expressions such that the values of~$w$ are exactly the words obtained, without any cancellation, by substituting admissible values into them.
Thus, questions about the values of~$w$ in the free group reduce to questions about words in the free monoid, in the spirit of Wicks' classical description of commutators~\cite{Wicks62}.
Besides commutators, this covers arbitrary quadratic twisted expressions, such as products of squares, twisted conjugates, and products of palindromes with respect to arbitrary antimorphic involutions.
In contrast to \EDTOL\ descriptions, such parametric descriptions are human-readable; moreover, our algorithm is simple enough to be implemented.

Our main application of parametric descriptions concerns palindrome length.
A word~$w$ is a \emph{palindrome} if it coincides with its reversal $w^\rho$, meaning that it reads the same from left to right and from right to left.
The palindrome length of a word, that is, the least number of palindromes into which it factors, has been studied extensively in combinatorics on words~\cite{FridPZ13, Saarela17}.
It can be computed in linear time~\cite{BorozdinKRS17} (see also~\cite{FiciGKK14, RubinchikS18, RubinchikS20}), and deciding whether it is at most a given bound is easily seen to be in \NL.
In free groups, palindrome length is defined in the same way, except that letters in a product of palindromes may cancel.
To make this precise, we take words over an alphabet~$A$ in which each letter~$a \in A$ has a formal inverse~$\bar a \in A$, and write $w \equiv w'$ if two words $w, w' \in A^\ast$ represent the same element of the free group~$F(A)$.
In this notation, the \emph{palindrome length} of a word $w \in A^\ast$, considered as an element of the free group $F(A)$, is given by
\[
  \pl(w) = \min \{\, n \geq 0 : w \equiv w_1 \cdots w_n \text{ for palindromes } w_1, \dotsc, w_n \,\}.
\]

Cancellations can make $\pl(w)$ arbitrarily smaller than the palindrome length of~$w$ as a word.
For instance, as observed by Saarela~\cite{Saarela17}, if $u = a_1 a_2 \cdots a_n$ and the letters $a_1, \dotsc, a_n, b_1, b_2$ are pairwise distinct, then the word $u b_1 b_2 u^\rho$ has no palindrome factor of length greater than one, so its palindrome length as a word is $2n + 2$, whereas the identity
\[
  u b_1 b_2 u^\rho \equiv (u b_1 u^\rho)\,(\bar u^\rho \bar u)\,(u b_2 u^\rho)
\]
yields $\pl(u b_1 b_2 u^\rho) \leq 3$.
Because of such cancellations, it is not even clear a priori whether $\pl$ is computable; this question was raised by Bardakov, Shpilrain, and Tolstykh~\cite[Problem~2]{BardakovST05},\footnote{See also \cite[Problem~16.9]{KhukhroM26}.} who had already shown that $\pl$ is unbounded on nonabelian free groups.
Recently, the second author~\cite{Frid25} characterized the freely reduced words $w$ with $\pl(w) \leq 2$ and those with $\pl(w) \leq 3$ by an extensive case analysis, and conjectured a Catalan structure for larger bounds.

In this article, we approach palindrome length through equations, as $\pl(w) \leq n$ holds if and only if $x_1 \cdots x_n = w$ has a solution in which the values of all variables $x_i$ are palindromes.
Owing to the palindrome constraint, this is a twisted equation, which moreover is quadratic.
In particular, \cref{pro:twisted-equations} immediately shows that $\pl$ is computable in \PSPACE\ (\cref{cor:palindrome-length-pspace}).
Computing the parametric descriptions for the expressions $x_1 \cdots x_n$ with an implementation of our algorithm, we arrived at a characterization of the words of bounded palindrome length in terms of \emph{patterns}, which describe how palindromes are nested by twisted conjugations.
We show that, for $n \geq 1$, a freely reduced word $w$ satisfies $\pl(w) \leq n$ if and only if it is obtained, without cancellation, from one of the patterns of rank~$n$ (\cref{thm:palindrome-length}).
This extends the characterizations for $n = 2, 3$ due to the second author~\cite{Frid25} and confirms the conjectured Catalan structure.
Although our proof of \cref{thm:palindrome-length} is direct, it closely follows the idea of symbolic reduction.

We also use our characterization to show that the words $w$ with $\pl(w) \leq n$, as well as the freely reduced words among them, form context-free languages, which can be described by a uniform family of grammars (\cref{pro:palindrome-length-language}).
From this, we conclude that the function $\pl$ is computable in $\LOGCFL \subseteq \PTIME$ (\cref{thm:palindrome-length-compute}) and, more concretely, in cubic time (\cref{cor:palindrome-length-cubic}).

\paragraph*{Outline.}
\Cref{sec:preliminaries} fixes notation, \cref{sec:twisted} introduces twisted equations, \cref{sec:quadratic} presents our algorithm for quadratic twisted equations, and \cref{sec:palindrome-length} applies our results to palindrome length.

\paragraph*{Concurrent Work.}
During the preparation of this article, several preprints~\cite{Jayadevan26, Liao26, KinyonU26} have proposed algorithms computing $\pl(w)$ in cubic time, based on a theorem of Dyer~\cite{Dyer01} on reflection length in Coxeter groups, and, in one case, the characterization of \cref{thm:palindrome-length} (see \cref{rem:coxeter}).
These works are specific to palindrome length and do not concern twisted equations in general.

Our results on palindrome length were obtained independently in August 2026, before these preprints appeared.
Our proof of \cref{thm:palindrome-length} is direct, and we additionally obtain context-free grammars generating the words $w$ with $\pl(w) \leq n$, and an algorithm computing~$\pl(w)$ in \LOGCFL.

\paragraph*{Disclosure of AI Use.}
Generative AI tools (large language models) were used for writing, that is, for drafting and revising expository passages, and for checking the manuscript for consistency.
All results and proofs are due to the authors, who take full responsibility for the content.

\section{Preliminaries}\label{sec:preliminaries}

In this section, we introduce notation and recall the basic notions needed in the remainder of the paper.
Since our results primarily concern twisted equations, in which the twists are given by the action of a finite group of symmetries, we work in an equivariant setting throughout.

Formally, we fix a finite group $G$ together with a homomorphism $\sgn \colon G \to \{\pm 1\}$.
The elements of $G$ are the \emph{symmetries} that we are interested in, and we call a symmetry $\alpha \in G$ \emph{positive} or \emph{negative} according to its \emph{signum} $\sgn(\alpha)$.
Consequently, the group of symmetries decomposes into the disjoint union $G = G_+ \cup G_-$, with $G_+ G_+, G_-G_- \subseteq G_+$ and $G_- G_+, G_+G_- \subseteq G_-$, where
\[
  G_+ = \{ \alpha \in G : \sgn(\alpha) = +1 \}
  \quad\text{and}\quad
  G_- = \{ \alpha \in G : \sgn(\alpha) = -1 \}.
\]
Throughout, we use overlines to denote inverses in groups; for instance, $\bar\alpha$ denotes the inverse of a symmetry $\alpha \in G$.
Note that, for every symmetry $\alpha \in G$, we have $\sgn(\bar \alpha) = \sgn(\alpha)$.

The signum dictates how a symmetry may act: positive symmetries act covariantly, that is, by automorphisms, whereas negative symmetries act contravariantly, that is, by anti-automorphisms.
We make this precise below for monoids equipped with a compatible $G$-action; the guiding examples are the free monoid $A^\ast$ and the free group $F(A)$ over a $G$-set~$A$.

\begin{example}\label{exm:palindrome-symmetries}
  In our main application, $G$ is generated by two distinct commuting involutions $\iota$ and~$\rho$, both of which are negative.
  Thus $G = \{1, \iota, \rho, \iota\rho\}$ with $G_+ = \{1, \iota\rho\}$ and $G_- = \{\iota, \rho\}$.

  On a free group $F(A)$, the involution $\iota$ acts by inversion and the involution $\rho$ acts by reversal of representing words over~$A$.
  Thus $\iota\rho = \rho\iota$ acts by taking inverses letterwise.
\end{example}

\subsection{\texorpdfstring{$G$-Sets and $G$-Monoids}{G-Sets and G-Monoids}}

All actions of the group $G$ of symmetries are right actions.
A \emph{$G$-set} is thus a set $X$ in which each $x \in X$ and $\alpha \in G$ determine an element $x^\alpha \in X$, subject to $x^{\alpha\beta} = (x^\alpha)^\beta$ and $x^1 = x$.

A map $f \colon X \to Y$ between $G$-sets is \emph{$G$-equivariant} if it is compatible with the two actions, meaning that $f(x^\alpha) = f(x)^\alpha$ for all $x \in X$ and $\alpha \in G$.
A subset $X \subseteq Y$ of a $G$-set $Y$ is called \emph{$G$-invariant} if $x^\alpha \in X$ for all $x \in X$ and $\alpha \in G$; such a subset is itself a $G$-set with respect to the restricted action, and its inclusion is a $G$-equivariant map.
Every subset $X \subseteq Y$ is contained in a unique minimal $G$-invariant subset, which is given by $X^G \coloneqq \{ x^\alpha : x \in X, \alpha \in G \}$.

For a singleton we write $x^G \coloneqq \{x\}^G$ and call this set the \emph{orbit} of $x$; we write $X/G$ for the set of orbits of~$X$.
As a $G$-set, $x^G$ is determined up to isomorphism by the \emph{stabilizer} of $x$, that is, by the subgroup $G_x \coloneqq \{ \alpha \in G : x^\alpha = x \} \subseteq G$ of those symmetries that fix $x$.
Indeed, $G_x \alpha \mapsto x^\alpha$ defines an isomorphism of $G$-sets from the set of right cosets $G_x \backslash G$, on which $G$ acts by right multiplication, onto $x^G$.
Moreover, stabilizers of elements in the same orbit are conjugate: we have $G_{x^\alpha} = G_x^\alpha$ for all $\alpha \in G$, where $H^\alpha \coloneqq \bar\alpha H \alpha$.

\medskip

Suppose that $M$ is a monoid and, at the same time, a $G$-set.
We call $M$ a \emph{$G$-monoid} if these two structures are compatible in the sense that, firstly, $1^\alpha = 1$ holds for all $\alpha \in G$, where $1 \in M$ is the neutral element, and, secondly, that
\[
  (uv)^\alpha = u^\alpha v^\alpha \text{ if } \alpha \in G_+
  \quad\text{and}\quad
  (uv)^\alpha = v^\alpha u^\alpha \text{ if } \alpha \in G_-
\]
for all $u,v \in M$ and $\alpha \in G$.
Thus, the map $u \mapsto u^\alpha$ is an automorphism of $M$ whenever $\alpha \in G_+$, and an anti-automorphism of $M$ whenever $\alpha \in G_-$.
A $G$-equivariant homomorphism $\varphi \colon M \to N$ between $G$-monoids is called a \emph{$G$-homomorphism}.
In this case we have $G_u \subseteq G_{\varphi(u)}$ for all $u \in M$.

\medskip

Elements of a $G$-monoid $M$ that are left unchanged by a negative symmetry play a special role throughout, and we give them a name.
We say that $u \in M$ is \emph{palindromic} if $u^\alpha = u$ for some $\alpha \in G_-$ or, equivalently, if $G_u \cap G_- \neq \emptyset$; otherwise, we say that $u$ is \emph{nonpalindromic}.

\subsection{Free Monoids and Groups}

For every $G$-set $A$, the \emph{free monoid} $A^\ast$, comprising all words over the alphabet $A$, carries a natural $G$-monoid structure: for a word $w = a_1 \cdots a_n \in A^\ast$ with $a_1, \dotsc, a_n \in A$ and $\alpha \in G$, we set
\[
  w^\alpha = a_1^\alpha \cdots a_n^\alpha \text{ if } \alpha \in G_+
  \quad\text{ and }\quad
  w^\alpha = a_n^\alpha \cdots a_1^\alpha \text{ if } \alpha \in G_-.
\]
With this structure, the free monoid $A^\ast$ has the expected universal property: every $G$-equivariant map $\varphi \colon A \to M$ to a $G$-monoid $M$ extends uniquely to a $G$-homomorphism $\hat\varphi \colon A^\ast \to M$.

As usual, $\abs{w}$ denotes the length of a word $w \in A^\ast$ and, for a letter $a \in A$, $\abs{w}_a$ denotes the number of occurrences of $a$ in $w$; thus, $\abs{w} = \sum_{a \in A} \abs{w}_a$.
Clearly, $\abs{w^\alpha} = \abs{w}$ for all $\alpha \in G$, implying that the $G$-action is compatible with factorizations:
if $w = uv$ satisfies $w^\alpha = w$ for some $\alpha \in G_+$, then $u^\alpha = u$ and $v^\alpha = v$; hence, $G_w \cap G_+ = G_u \cap G_v \cap G_+$.
Similarly, if $w = uqv$ with $\abs{u} = \abs{v}$ satisfies $w^\alpha = w$ for some $\alpha \in G_-$, then $u^\alpha = v$, $v^\alpha = u$, and $q^\alpha = q$.
We use these simple observations throughout without further reference.

\medskip

For the remainder of this section, let $A$ be a $G$-set, and let $\iota \in G_-$ be a fixed central involution; that is, $\iota \alpha = \alpha \iota$ for all $\alpha \in G$ and $\iota^2 = 1$.
We write $\bar a \coloneqq a^\iota$ for a letter $a \in A$ and treat $\bar a$ as a formal inverse of $a$; this notation is extended to words by $\bar w \coloneqq w^\iota$ for $w \in A^\ast$.
Since $\iota$ is a central involution, we obtain $\overline{uv} = \bar v \bar u$,
  $\bar{\bar w} = w$, and $\overline{w^\alpha} = \bar w^\alpha$ for all $u,v,w \in A^\ast$ and $\alpha \in G$.

The \emph{free group} $F(A)$ is the quotient of the free monoid $A^\ast$ by the smallest congruence $\equiv$ such that $w \bar w \equiv 1$ holds for all $w \in A^\ast$.
Strictly speaking, $F(A)$ is the free group over the alphabet $A$ with involution $\iota$; as this involution is fixed, we omit it from the notation.\footnote{If $\bar a \neq a$ for all $a \in A$, then $F(A)$ is a free group in the usual sense, with $A$ comprising both the free generators and their (formal) inverses. In general, every letter $a \in A$ with $\bar a = a$ yields an element of order two; thus, $F(A)$ is a free product of infinite cyclic groups and groups of order two.}

We write $\pi_A \colon A^\ast \to F(A)$ for the canonical homomorphism mapping each word to its $\equiv$-class.
Since the congruence $\equiv$ is $G$-invariant, the free group $F(A)$ inherits a natural $G$-action from the free monoid $A^\ast$, with respect to which $\pi_A$ is a $G$-homomorphism.
As in the nonequivariant setting, the congruence class of a word $w \in A^\ast$ contains a unique \emph{freely reduced} word $\red(w) \in A^\ast$, meaning a word without factors of the form $a \bar a$ with $a \in A$; moreover, $\red(w)$ is obtained from $w$ by successively deleting arbitrary such factors.
Clearly, $\red(w^\alpha) = \red(w)^\alpha$ for all $w \in A^\ast$, $\alpha \in G$.

\subsection{Formal Languages}

We assume familiarity with basic concepts of formal language theory, in particular with regular and context-free languages; see, e.g., \cite{HopcroftU79} for background.
Familiarity with \EDTOL\ languages is not required, as we only use their characterization in terms of rationally controlled endomorphisms due to Asveld~\cite{Asveld77}; see, e.g., \cite{RozenbergS80} for further background.

Suppose that $A$ is a finite $G$-set.
As is well known, a language $L \subseteq A^\ast$ is regular if and only if it is \emph{recognized} by some homomorphism $\mu \colon A^\ast \to N$ to a finite monoid $N$, in the sense that $L = \mu^{-1}(\mu(L))$; see, e.g., \cite{Pin97}.
The following lemma shows that, in this characterization, we may additionally assume $N$ to be a $G$-monoid and $\mu$ to be a $G$-homomorphism.

\begin{lemma}\label{lem:equivariant-recognition}
    Let $N$ be a (finite) monoid.
    Then there is a surjective homomorphism \mbox{$\pi_N \colon N_G \to N$} from a (finite) $G$-monoid $N_G$ that is universal in the following sense: if $\varphi \colon M \to N$ is a homomorphism from a $G$-monoid $M$, then $\varphi = \pi_N \circ \varphi_G$ for a unique $G$-homomorphism $\varphi_G \colon M \to N_G$.
\end{lemma}
\begin{proof}
    We define $N_G$ to be the direct product $\prod_{\alpha \in G} N_\alpha$, where the factor $N_\alpha$ is $N$ itself if $\alpha \in G_+$, and the opposite monoid $N^{\mathrm{op}}$ if $\alpha \in G_-$.
    Since $G$ is finite, the monoid $N_G$ is finite whenever $N$ is.
    The homomorphism $\pi_N \colon N_G \to N$ is the projection onto the factor $N_1 = N$ at $1 \in G$.

    The group $G$ acts on $N_G$ by permuting the factors $N_\alpha$.
    More precisely, we stipulate that for all elements $u \in N_G$ and all $\alpha,\beta \in G$, the $\alpha\beta$-component of $u^\beta$ is given by the $\alpha$-component of $u$.
    A routine computation, which we omit, shows that this action turns $N_G$ into a $G$-monoid.

    Let $\varphi \colon M \to N$ be a homomorphism from a $G$-monoid $M$.
    Then, for $\alpha \in G$, $\varphi_\alpha(u) \coloneqq \varphi(u^{\bar\alpha})$ defines a homomorphism $\varphi_\alpha \colon M \to N_\alpha$.
    As $N_G$ is the direct product of the~$N_\alpha$, these maps combine into a homomorphism \mbox{$\varphi_G \colon M \to N_G$}.
    The latter clearly satisfies $\varphi = \pi_N \circ \varphi_G$.
    Since $\beta \overline{\alpha \beta} = \bar \alpha$, we have $\varphi_{\alpha \beta}(u^\beta) = \varphi_{\alpha}(u)$ for all $u \in M$ and $\alpha, \beta \in G$, showing that $\varphi_G$ is $G$-equivariant.

    To show uniqueness of $\varphi_G$, let us consider any $G$-homomorphism $\psi \colon M \to N_G$ with $\varphi = \pi_N \circ \psi$.
    Then, for every $u \in M$, the $\alpha$-component of $\psi(u) = \psi(u^{\bar \alpha \alpha}) = \psi(u^{\bar \alpha})^\alpha$ must agree with the $1$-component of $\psi(u^{\bar \alpha})$, that is, with $\pi_N(\psi(u^{\bar \alpha})) = \varphi(u^{\bar \alpha}) = \varphi_\alpha(u)$.
    Hence, $\psi = \varphi_G$.
\end{proof}

The concept of regular languages admits several generalizations to arbitrary monoids, of which we now recall the two most important ones.
For simplicity, we restrict ourselves to finitely generated monoids; see, e.g., \cite{Eilenberg74} for a general treatment.
In the following, let $M$ be a monoid and $\pi \colon A^\ast \to M$ a surjective homomorphism where, as before, $A$ is finite.

\begin{itemize}
\item A subset $R \subseteq M$ is \emph{recognizable} if it is recognized by a homomorphism $\mu \colon M \to N$ to a finite monoid $N$, meaning that $R = \mu^{-1}(\mu(R))$.\footnote{Note that \cref{lem:equivariant-recognition} applies at this level of generality.}
It is easy to see that a subset $R \subseteq M$ is recognizable if and only if its full preimage $\pi^{-1}(R) \subseteq A^\ast$ is a regular language.

\item A subset $R \subseteq M$ is \emph{rational} if $R = \pi(L)$ for some regular language $L \subseteq A^\ast$; this notion does not depend on the choice of $\pi$.
\end{itemize}

For the free group $F(A)$, rational subsets are characterized in terms of their freely reduced representatives by the following result of Benois~\cite{Benois69}.\footnote{Benois only considers the case that $\bar a \neq a$ for all $a \in A$; the proof carries over verbatim.}

\begin{theorem}[Benois]\label{thm:benois}
    For every regular language $L \subseteq A^\ast$, the language $\red(L) \subseteq A^\ast$ is regular.
    Consequently, a subset $R \subseteq F(A)$ is rational if and only if $\red(\pi_A^{-1}(R)) \subseteq A^\ast$ is regular.
\end{theorem}

As freely reduced words play an important role in our treatment, we fix a $G$-homomorphism $\mu_A \colon A^\ast \to N_A$ recognizing the language $\red(A^\ast) \subseteq A^\ast$.
Its codomain is the finite $G$-monoid $N_A = \{1, 0\} \cup (A \times A)$, in which $1$ is the neutral element, and $0$ is the zero element.
The remaining products are given by
\[
  (a_1, a_2)(a_3, a_4) =
  \begin{cases}
    (a_1, a_4) & \text{if } \bar a_2 \neq a_3, \\
    0 & \text{otherwise.}
  \end{cases}
\]
The $G$-action is defined by $(a_1, a_2)^\alpha = (a_1^\alpha, a_2^\alpha)$ for~$\alpha \in G_+$, $(a_1, a_2)^\alpha = (a_2^\alpha, a_1^\alpha)$ for~$\alpha \in G_-$, as well as $1^\alpha = 1$, $0^\alpha = 0$ for $\alpha \in G$.
On letters, $\mu_A$ is given by $a \mapsto (a, a)$.
Intuitively, $\mu_A$ remembers the first and last letter of a nonempty, freely reduced word.
It is easy to see that a word $w \in A^\ast$ is nonempty if and only if $\mu_A(w) \neq 1$, and that $w \in A^\ast$ is freely reduced if and only if $\mu_A(w) \neq 0$.

\section{Twisted Equations in Free Groups}\label{sec:twisted}

The majority of this work deals with expressions over constants and variables in a free group, subject to the action of a finite group of symmetries and to rational constraints on the variables.
In \cref{sec:expressions}, we formalize such expressions, the contexts in which they live, and substitutions between these contexts; in \cref{sec:twisted-equations}, we then turn to twisted equations in free groups.

Throughout, we fix a finite group $G$ of positive and negative symmetries as in \cref{sec:preliminaries}, together with a central negative involution $\iota \in G_-$.
For all results regarding palindrome length, $G$ and $\iota$ are as in \cref{exm:palindrome-symmetries}.
Constants are taken from a fixed finite $G$-set $A$, whereas variables are taken from finite $G$-sets $X$ disjoint from $A$, which vary throughout.
Elements of the free group~$F(A)$ and of the free groups $F(A \cup X)$ are represented by freely reduced words over $A$ and $A \cup X$, respectively.

\subsection{Expressions and Substitutions}\label{sec:expressions}

Intuitively, an expression is a word over constants and variables, where each variable stands for an unknown element of the free group $F(A)$, possibly subject to a rational constraint.
We represent the values of variables by freely reduced words, and encode their constraints by means of a finite $G$-monoid.
To this end, we fix a $G$-homomorphism $\mu_\bot \colon A^\ast \to N$ to such a monoid $N$.
We assume that $\mu_\bot$ refines the $G$-homomorphism $\mu_A \colon A^\ast \to N_A$ from \cref{sec:preliminaries}, meaning that $\mu_A = \nu \circ \mu_\bot$ for some fixed $G$-homomorphism $\nu \colon N \to N_A$.
Thus, $\mu_\bot$ provides the information needed to express constraints, whereas $\nu \circ \mu_\bot$ detects cancellations.

\begin{definition}
  A \emph{context} $C = (X, \mu)$ consists of a finite $G$-set $X$ with $X \cap A = \emptyset$, its set of \emph{variables}, and a $G$-homomorphism $\mu \colon (A \cup X)^\ast \to N$, its \emph{constraints}, such that
  \[
    \mu(a) = \mu_\bot(a)
    \text{ for all }
    a \in A
    \quad\text{and}\quad
    (\nu\circ\mu)(x) \neq 0
    \text{ for all }
    x \in X.
  \]
  The \emph{terminal context} is the unique context  $C_\bot = (X_\bot, \mu_\bot)$ without variables (i.e., $X_\bot = \emptyset$).
\end{definition}

The constraints $\mu$ are determined by their values on the variables; in fact, $\mu(x)$ can be chosen freely for a single variable $x$ from each orbit, subject only to $G_x \subseteq G_{\mu(x)}$ and $(\nu \circ \mu)(x) \neq 0$.
The latter condition will ensure that variables only take freely reduced words as values (see below).
Similarly, stipulating $(\nu \circ \mu)(x) = 1$ forces the value of a variable $x$ to be the empty word.
We call such variables \emph{singular}, and all other variables \emph{nonsingular}; moreover, we say that a context is \emph{nonsingular} if all of its variables are.

\begin{definition}\label{def:expression}
  An \emph{expression} over a context $C = (X, \mu)$ is a freely reduced word $w \in (A \cup X)^\ast$.
\end{definition}

Thus, the expressions over $C$ can be identified with the elements of the free group $F(A \cup X)$; under this identification, the expressions over $C_\bot$ are the elements of $F(A)$.

\begin{example}\label{exm:palindrome-expression}
  Let $G$ be as in \cref{exm:palindrome-symmetries}, and let $\mu_\bot = \mu_A$.
  Further, consider the $G$-set of variables $X_n = \{x_1, \bar x_1, \dotsc, x_n, \bar x_n\}$ with $\rho$ acting trivially, so that each $x_i$ represents a palindrome.
  Then the word $x_1 \cdots x_n$ is an expression over each context $(X_n, \mu)$, where the constraints $\mu \colon (A \cup X_n)^\ast \to N_A$ range over the finitely many suitable extensions of $\mu_A$; together, they serve as a template for the products of at most $n$ palindromes in $F(A)$.
\end{example}

\begin{definition}
  A $G$-homomorphism $\sigma \colon (A \cup X)^\ast \to (A \cup X')^\ast$ is a \emph{substitution} between contexts $(X, \mu)$ and $(X', \mu')$, written $\sigma \colon (X, \mu) \to (X', \mu')$, provided that $\sigma(a) = a$ for all $a \in A$, that $\mu = \mu' \circ \sigma$, and that every $x' \in X'$ occurs in the value $\sigma(x)$ of some $x \in X$.
\end{definition}

The second condition ensures that substitutions respect constraints, and the third condition states that they cannot introduce superfluous variables.
Every substitution $\sigma \colon C \to C'$ induces a map $w \mapsto \sigma_\ast(w) \coloneqq \red(\sigma(w))$ from expressions over $C$ to expressions over $C'$; in terms of free groups, $\sigma_\ast$ is simply the $G$-homomorphism $F(A \cup X) \to F(A \cup X')$ induced by $\sigma$.
Clearly, identities and compositions of substitutions are substitutions, and $(\sigma' \circ \sigma)_\ast = \sigma'_\ast \circ \sigma_\ast$.

As a map, a substitution $\sigma$ is fully determined by its restriction to the set of variables~$X$; in fact, it suffices to specify the images of a single variable $x$ from each orbit.
Conversely, any choice of words $\sigma(x) \in (A \cup X')^\ast$ with $G_x \subseteq G_{\sigma(x)}$ and $\mu'(\sigma(x)) = \mu(x)$ for these representatives extends uniquely to a $G$-homomorphism satisfying the first two conditions above.
We frequently specify substitutions in this way; for instance, we speak of the substitution $\sigma$ \emph{defined} by $x \mapsto u x \bar v$, with the understanding that $\sigma(y) = y$ for all $y \in X \setminus x^G$.

Substitutions $\sigma \colon C \to C_\bot$ to the terminal context are called \emph{terminal}.
We write $\tsub(C)$ for the set of these substitutions.
By definition, a terminal substitution $\sigma$ is simply a $G$-equivariant assignment of words $\sigma(x) \in A^\ast$ to the variables $x \in X$ such that $\mu_\bot(\sigma(x)) = \mu(x)$.

We think of the words $\sigma(x)$ with $\sigma \in \tsub(C)$ as the possible \emph{values} of a variable $x \in X$; they are freely reduced since $\mu_A(\sigma(x)) = (\nu \circ \mu)(x) \neq 0$.
Hence, by \cref{thm:benois}, requiring $x$ to take values in a rational subset $R \subseteq F(A)$ amounts to choosing $\mu(x)$ from $\mu_\bot(\red(\pi_A^{-1}(R))) \subseteq N$, provided that $\mu_\bot$ recognizes the regular language $\red(\pi_A^{-1}(R)) \subseteq A^\ast$.
More generally, we think of the words $\sigma_\ast(w)$ with $\sigma \in \tsub(C)$ as the possible values of an expression $w$ over $C$.

Not every context admits a terminal substitution, and we call the ones that do \emph{feasible}; thus, $C$ is feasible if $\tsub(C) \neq \emptyset$.\footnote{In our setting, the group $G$ and the $G$-homomorphism $\mu_\bot \colon A^\ast \to N$ are fixed.
As feasibility of $C = (X, \mu)$ depends only on the pairs $(G_x, \mu(x))$ associated with the variables $x \in X$, it is fairly easy to decide.}
For every substitution $\tau \colon C \to C'$ and every $\sigma' \in \tsub(C')$, we have $\sigma' \circ \tau \in \tsub(C)$; that is, $\tau$ induces a map $\tau^\ast \colon \tsub(C') \to \tsub(C)$ where $\tau^\ast(\sigma') \coloneqq \sigma' \circ \tau$.
For instance, if $C$ is feasible, then every singular variable $x \in X$ must satisfy $\mu(x) = \mu_\bot(1) = 1$, as its only possible value is the empty word~$1$; hence, the substitution $\varepsilon$ defined by $x \mapsto 1$ for all singular variables $x$ leads from $C$ to the nonsingular context $C'$ obtained by removing these variables, and it induces a bijection $\varepsilon^\ast \colon \tsub(C') \to \tsub(C)$.

\subsection{General Twisted Equations}\label{sec:twisted-equations}

With the notions of expressions and substitutions at hand, twisted equations and their solutions can now be easily defined as follows.

\begin{definition}
  A \emph{twisted equation $u = v$} in the free group $F(A)$ consists of two expressions, $u$~and~$v$, over a common context $C$.
  A \emph{solution} of $u = v$ is a terminal substitution $\sigma \in \tsub(C)$ with $\sigma_\ast(u) = \sigma_\ast(v)$.
  We write $\sol(u = v) \subseteq \tsub(C)$ to denote the set of all solutions of $u = v$.
\end{definition}

Besides a variable $x \in X$, the expressions $u$ and $v$ may contain its images $x^\alpha$ under arbitrary symmetries $\alpha \in G$, and every solution $\sigma$ satisfies $\sigma(x^\alpha) = \sigma(x)^\alpha$; in this sense, the equation is twisted by the symmetries in $G$.
Moreover, by choosing appropriate constraints, we can require the variables to take values in rational subsets of $F(A)$; see \cref{sec:expressions}.

\begin{example}\label{exm:palindrome-equation}
  Suppose that $u \in A^\ast$ is freely reduced.
  Then $\pl(u) \leq n$ if and only if the twisted equation $x_1 \cdots x_n = u$ has a solution over one of the contexts $(X_n, \mu)$ from \cref{exm:palindrome-expression}.
  Hence, if we replace the right-hand side by a variable $y$ on which $G$ acts freely, the values $\sigma(y)$ of all solutions $\sigma \in \sol(x_1 \cdots x_n = y)$ are exactly the freely reduced words $w$ with $\pl(w) \leq n$.
\end{example}

In the following, we consider \emph{systems of twisted equations} in $F(A)$, meaning finite sets
\[
  S = \{u_1 = v_1, \dotsc, u_n = v_n\}
\]
of twisted equations over a common context $C$, which must all be satisfied simultaneously.
Accordingly, the solution set of such a system $S$ is defined as $\sol(S) \coloneqq \bigcap_i \sol(u_i = v_i)$.

Fixing representatives $x_1, \dotsc, x_k$ of the orbits of the variables of $C$, we encode every solution~$\sigma$ by the word $\sigma(x_1) \# \cdots \# \sigma(x_k)$, where $\#$ is a new letter; by an \emph{\EDTOL\ description} of $\sol(S)$, we mean an \EDTOL\ system generating the language of these encodings.

\medskip

Twisted equations arise naturally when solving equations in virtually free groups, and they have been studied in this context before; see, e.g., \cite{DahmaniG10, DiekertE20}.
There, however, the twists are given by automorphisms only: a finite group acts by automorphisms on a free group or, in the combinatorial setting of~\cite{DiekertE20}, on a free monoid with involution.
The following reduction shows that admitting anti-automorphisms as well does not lead beyond this setting; it rests on the observation that, since $G_- = G_+ \iota$, every negative symmetry acts on $F(A)$ as a positive symmetry followed by inversion.

\begin{proposition}\label{pro:twisted-equations}
    On input of a system of twisted equations $S$ with rational constraints in the free group $F(A)$, one can decide in \PSPACE\ whether $\sol(S) \neq \emptyset$.
    Moreover, $\sol(S)$ has an \EDTOL\ description computable by a \PSPACE-transducer.
\end{proposition}
\begin{proof}
  Let $H = \{ s_\alpha : \alpha \in G_+ \}$ be an isomorphic copy of $G_+$, with multiplication $s_\alpha s_\beta = s_{\alpha \beta}$; and let $F(A) \rtimes H$ be the semidirect product in which $\overline{s_\alpha}\,g\,s_\alpha = g^\alpha$ for all $g \in F(A)$ and $\alpha \in G_+$.
  Since the group $F(A)$ has finite index in $F(A) \rtimes H$, the latter is a finitely generated virtually free group.

  Suppose that $S$ consists of twisted equations over the context $C = (X, \mu)$.
  We then read $S$ as a system $S'$ of ordinary equations over $F(A) \rtimes H$, in which each $x \in X$ is an unknown in its own right and each $a \in A$ denotes the constant $\pi_A(a)$.
  To account for the $G$-action and the constraints, we add to $S'$ the equations
  \[
    x \bar x = 1
    \quad\text{and}\quad
    \overline{s_\alpha} \, x \, s_\alpha = x^\alpha
    \qquad\text{for all } x \in X \text{ and } \alpha \in G_+,
  \]
  together with the rational constraints $x \in R_x \coloneqq \pi_A(\mu_\bot^{-1}(\mu(x))) \subseteq F(A)$ for all $x \in X$.

  Clearly, $S'$ can be computed from $S$ in polynomial time.
  As $\iota$ acts on $F(A)$ by inversion and $G$ is generated by $G_+ \cup \{\iota\}$, the added equations state precisely that the values assigned to the unknowns depend $G$-equivariantly on $x \in X$.
  Moreover, since $(\nu \circ \mu)(x) \neq 0$, all words in $\mu_\bot^{-1}(\mu(x))$ are freely reduced, so that $\pi_A$ maps this language bijectively onto the set $R_x \subseteq F(A) \subseteq F(A) \rtimes H$.
  Therefore, the solutions of $S$ are in bijection with those of $S'$.

  The decidability statement now follows from the result of the first author and Elder~\cite{DiekertE20} that, in every fixed finitely generated virtually free group, satisfiability of systems of equations with rational constraints is decidable in \PSPACE.
  Likewise, the \EDTOL\ description can be obtained using the corresponding result in~\cite{DiekertE20}.
\end{proof}

Combining \cref{exm:palindrome-equation} and \cref{pro:twisted-equations} immediately yields the following result.

\begin{corollary}\label{cor:palindrome-length-pspace}
  The function $\pl$ is computable in \PSPACE.
  Moreover, for every $n \geq 0$, the freely reduced words $w$ with $\pl(w) \leq n$ form an \EDTOL\ language, with a \PSPACE-computable description.
\end{corollary}

\begin{remark}\label{rem:coxeter}
  For palindrome length, the group constructed in the proof of \cref{pro:twisted-equations} is a Coxeter group.
  To see this, suppose that $\rho$ acts trivially on $A$, as in \cref{sec:palindrome-length}.
  Then $H = \{1, s\}$ with $s \coloneqq s_{\iota\rho}$, and the involution $s$ satisfies $s a s = \bar a$ for all $a \in A$.
  Hence, the resulting group $W = F(A) \rtimes H$ is generated by the involutions $s$ and $as$, where $a$ ranges over representatives of the orbits of $\iota$ on $A$.
  In fact, $W$ is the Coxeter group on these generators in which $s$ and $as$ commute if $\bar a = a$, and all other products of two distinct generators have infinite order; in particular, if $\iota$ acts freely on $A$, then $W$ is a free product of groups of order two, that is, a universal Coxeter group.

  Since $g \in F(A)$ is a palindrome\footnote{We call $g \in F(A)$ a palindrome if its freely reduced representative over $A$ is.} if and only if $g^{\iota\rho} = \bar g$ (cf.~\cite[Lemma~1.4]{BardakovST05}), that is, if and only if $(gs)^2 = 1$, one easily verifies that $g$ is a palindrome if and only if $gs$ is a reflection, that is, a conjugate of one of the above generators.
  Since $(g_1 s)(g_2 s) = g_1 \bar g_2$ for all palindromes $g_1, g_2 \in F(A)$, and inverses of palindromes are palindromes, we obtain
  \[
    \pl(g) = \min\{\ell_T(g), \ell_T(gs)\}
  \]
  for all $g \in F(A)$, where $\ell_T(x)$ denotes the reflection length of~$x \in W$, that is, the least number of reflections whose product is $x$.
  By a theorem of Dyer~\cite{Dyer01}, the reflection length $\ell_T(x)$ is the least number of letters that must be deleted from a reduced expression for $x$ over the above generators to obtain an expression for the identity.
  If $\iota$ acts freely on $A$, the latter can be computed in cubic time by dynamic programming over noncrossing matchings of equal letters; this is the route taken by the concurrent works~\cite{Jayadevan26, Liao26, KinyonU26}.
  Dyer's theorem itself determines only the palindrome length of a given element.
  From the structure of optimal matchings, Liao~\cite{Liao26} moreover derives the characterization of \cref{thm:palindrome-length} in the case that $\iota$ acts freely on $A$.
  In contrast, our proof of \cref{thm:palindrome-length} is direct and does not use Coxeter groups.
  Moreover, our \cref{thm:parametric} yields parametric descriptions for arbitrary quadratic twisted expressions.
\end{remark}

Both statements of \cref{cor:palindrome-length-pspace} will be sharpened in \cref{sec:palindrome-length}: we show that $\pl$ can be computed in \LOGCFL\ (\cref{thm:palindrome-length-compute}), and that the languages above are also context-free, with an explicit uniform family of grammars in place of \PSPACE-computable \EDTOL\ systems (\cref{lem:palindrome-length-grammar,pro:palindrome-length-language}).

\section{Quadratic Twisted Equations in Free Groups}\label{sec:quadratic}

By \cref{pro:twisted-equations}, systems of twisted equations in free groups can be solved in \PSPACE.
The known algorithms for solving equations in virtually free groups are, however, intricate: the algorithm of Dahmani and Guirardel~\cite{DahmaniG10} is based on the Rips classification of measured band complexes, whereas that of the first author and Elder~\cite{DiekertE20} relies on Je\.z's recompression technique~\cite{Jez16}.
Both approaches are of great theoretical value, but we are not aware of any implementation.
Restricting the class of equations affords simpler and, indeed, practical algorithms.
In this section, we present a simpler algorithm for solving quadratic twisted equations.
Its underlying structure goes back to Comerford and Edmunds~\cite{ComerfordE81}, who considered quadratic equations in free groups; here, we extend it to the twisted setting and to rational constraints.
Besides deciding solvability, our algorithm yields parametric descriptions of the values of quadratic expressions (\cref{thm:parametric}), and these will serve as our main tool for studying palindrome length in \cref{sec:palindrome-length}.

We first introduce quadratic expressions in \cref{sec:quadratic-expressions}.
In \cref{sec:symbolic-reduction}, we then make precise what it means to reduce an expression symbolically, that is, without knowing the values of its variables.
In \cref{sec:automaton}, we introduce an automaton whose transitions carry out symbolic reduction step by step by means of basic substitutions, and we derive our main applications.
The proof that symbolic reduction is complete is given in \cref{sec:completeness}.

\subsection{Quadratic Expressions and Splitting}\label{sec:quadratic-expressions}

Recall that we call a variable $x$ palindromic if $x^\alpha = x$ for some $\alpha \in G_-$.
Whether this is the case depends only on its orbit $x^G$, since $G_{x^\alpha} = G_x^\alpha$ and $\sgn(\bar\alpha\beta\alpha) = \sgn(\beta)$ for all $\alpha, \beta \in G$.
Accordingly, we call an orbit \emph{palindromic} if its elements are, and \emph{nonpalindromic} otherwise.

\begin{definition}\label{def:quadratic}
  An expression $w$ over a context $(X, \mu)$ is \emph{quadratic} if the following hold.
  \begin{itemize}
    \item If $x^G \subseteq X$ is nonpalindromic, then variables from $x^G$ occur at most twice in $w$ in total.
    \item If $x^G \subseteq X$ is palindromic, then variables from $x^G$ occur at most once in $w$ in total.
  \end{itemize}
\end{definition}

\begin{example}\label{exm:quadratic-general}
  For $G = \{1, \iota\}$, quadratic twisted equations include the usual quadratic equations in free groups; for instance, $x_1 x_1 \cdots x_n x_n = u$ and $\bar x_1 \bar y_1 x_1 y_1 \cdots \bar x_n \bar y_n x_n y_n = u$ express that $u$ is a product of $n$ squares or $n$ commutators, respectively.
  More generally, the twisted conjugacy equation $\bar x^\alpha v x = u$ with $\alpha \in G$ is quadratic, and so is $x_1 \cdots x_n = u$ if each variable $x_i$ is fixed by some negative symmetry, including the equation from \cref{exm:palindrome-equation}.
\end{example}

For palindrome length, we obtain the following quadratic expressions.

\begin{example}\label{exm:palindrome-quadratic}
The expression $x_1 \cdots x_n$ with variables $x_i = x_i^\rho$ from \cref{exm:palindrome-expression} is quadratic, since each of the palindromic orbits contributes exactly one occurrence.
Similarly, both of the expressions
\[
  x_1 \cdots x_n \bar u
  \quad\text{and}\quad
  x_1 \cdots x_n \bar y
\]
with constant $u \in \red(A^\ast)$ and variable $y \notin x_1^G \cup \cdots \cup x_n^G$ are quadratic.
\end{example}

The stricter bound for palindromic orbits in \cref{def:quadratic} is due to the fact that, under substitutions, each palindromic occurrence behaves like two nonpalindromic ones.
For example, if $x$ is nonpalindromic, then the substitution $\tau$ defined by $x \mapsto yx$ with a new nonpalindromic variable $y$ replaces each occurrence from $x^G$ by a word containing one occurrence from $y^G$.
If, however, $x^\alpha = x$ for some $\alpha \in G_-$, then the natural analogue, defined by $x \mapsto y x y^\alpha$ with $\alpha^2 \in G_y \subseteq G_+$, instead replaces each occurrence from $x^G$ by a word containing two occurrences from $y^G$.
In both cases, the bounds in \cref{def:quadratic} ensure that, for every quadratic expression $w$, the resulting expression $\tau_\ast(w)$ remains quadratic.

Using a variant of the latter, we can even eliminate palindromic variables: depending on its value, a palindromic variable $x$ can be split via a substitution defined by $x \mapsto y a y^\alpha$ with a new nonpalindromic variable $y$ and a suitable constant $a \in A \cup \{1\}$.

\begin{lemma}\label{lem:splitting}
  Suppose that $C = (X, \mu)$ is a context.
  Then, for every $\sigma \in \tsub(C)$, there exist a substitution $\tau \colon C \to C'$ to a nonsingular context $C' = (X', \mu')$ with $\abs{X'/G} \leq \abs{X/G}$ and without palindromic variables, as well as a terminal substitution $\sigma' \in \tsub(C')$ with $\sigma = \sigma' \circ \tau$.
  Moreover, we have $\abs{\tau_\ast(w)} \leq 3 \cdot \abs{w}$ for every expression $w$ over $C$, and $\tau_\ast$ preserves quadraticity.
\end{lemma}
\begin{proof}
  Suppose that $x \in X$ is palindromic, say $x^\alpha = x$ with $\alpha \in G_-$.
  Then we write $\sigma(x) = y' a y'^\alpha$ for some $y' \in A^\ast$ and $a \in A \cup \{1\}$ with $y'^{\alpha\alpha} = y'$ and $a^\alpha = a$.
  We also introduce a new variable $y$ with stabilizer $G_y = G_{y'} \cap G_+$; in particular, $y$ is nonpalindromic and $y^{\alpha\alpha} = y$.
  Let $\tau$ be the substitution defined by $x \mapsto y a y^\alpha$.
  It leads to a context $C' = (X', \mu')$, where $X' = (X \setminus x^G) \cup y^G$, and $\sigma' \in \tsub(C')$ with $\sigma = \sigma' \circ \tau$ and $\sigma'(y) = y'$.
  Clearly, $\abs{\tau_\ast(z)} \leq 3$ for all $z \in x^G$, while $\abs{\tau_\ast(z)} = 1$ for all $z \in X \setminus x^G$; in particular, $\abs{\tau_\ast(w)} \leq 3 \cdot \abs{w}$ holds for every expression $w$ over $C$.
  Moreover, if $w$ is a quadratic expression, then variables from $x^G$ occur at most once in $w$, so that variables from $y^G$ occur at most twice in $\tau_\ast(w)$; hence, the expression $\tau_\ast(w)$ is again quadratic.

  Repeating this for every palindromic orbit of $X$, and then eliminating all singular variables as in \cref{sec:expressions}, we obtain the desired substitution as a composition; since each variable of $X$ is affected by at most one of these splittings, the bounds carry over.
\end{proof}

\subsection{Symbolic Free Reduction}\label{sec:symbolic-reduction}

Recall that the values of an expression $w$ are the freely reduced words $\sigma_\ast(w) = \red(\sigma(w)) \in A^\ast$, with $\sigma \in \tsub(C)$ any terminal substitution from its context $C = (X, \mu)$.
The words $\sigma(w)$ themselves, however, are in general not freely reduced, and the difficulty of determining the values of $w$ lies essentially in their free reduction.
Our aim is to carry out this free reduction symbolically, that is, without knowledge of the terminal substitution $\sigma$; we refer to this as \emph{symbolic reduction}.
For this, we will make use of the fact that the constraints $\mu$ of $C$ already determine whether cancellations occur.
How much cancels, however, depends on the unknown $\sigma$; how to reveal this information is covered in \cref{sec:automaton}.

\begin{definition}\label{def:reduced}
  A word $w \in (A \cup X)^\ast$ is \emph{reduced} if $(\nu \circ \mu)(w) \neq 0$.\footnote{Strictly speaking, this notion depends on the context $C$ through its constraints $\mu$, and so do the notions introduced in \cref{def:syllable} below. Whenever the context is not clear, we say that $w$ is reduced \emph{over}~$C$.}
\end{definition}

For every terminal substitution $\sigma \in \tsub(C)$, we have $\mu_A \circ \sigma = \nu \circ \mu_\bot \circ \sigma = \nu \circ \mu$.
Since a word over $A$ is freely reduced if and only if its image under $\mu_A$ is nonzero, it follows that $\sigma(w)$ is freely reduced if and only if $w$ is reduced; in particular, $\sigma_\ast(w) = \sigma(w)$ for every reduced expression~$w$.
Note that, as expressions are freely reduced by definition, being reduced is an additional requirement: for instance, if $(\nu \circ \mu)(x) = (b, \bar a)$ for a variable $x \in X$ and $a, b \in A$, then the expression $xa$ is not reduced, and indeed every word $\sigma(xa)$ ends with $\bar a a$.

Since $0 \in N_A$ is a zero element, factors of reduced words are reduced.
Moreover, whether a product of two elements of $A \times A \subseteq N_A$ vanishes depends only on the second component of the first factor and the first component of the second factor.
Consequently, for $u,w,v \in (A \cup X)^\ast$ with $(\nu \circ \mu)(w) \neq 1$, the word $uwv$ is reduced if and only if both words $uw$ and $wv$ are~--~thus, being a reduced word is a local property.
This motivates the following definition.

\begin{definition}\label{def:syllable}
  A \emph{syllable} of $w \in (A \cup X)^\ast$ is a maximal nonempty reduced factor of~$w$, where maximality is with respect to inclusion.
  We write $\syl(w)$ for the number of syllables~of~$w$.
\end{definition}

Since every single symbol in $A \cup X$ is reduced, the syllables of a word $w$ cover $w$, and $w$ is reduced if and only if $\syl(w) \leq 1$.
In general, syllables may overlap: if $x \in X$ is singular, then the expression $\bar a x a$ with $a \in A$ has two overlapping syllables, viz.\ $\bar a x$ and $x a$.
By the following lemma, however, overlaps consist of singular variables only.

\begin{lemma}\label{lem:syllable-basics}
  Let $w \in (A \cup X)^\ast$ be a word over $C = (X, \mu)$.
  Then the following hold.
  \begin{enumerate}[label=(\arabic*)]
    \item\label{lem:syllable-basics-1}
      The intersection of two distinct syllables of $w$ is a word in singular variables.
    \item\label{lem:syllable-basics-2} The number $\syl(w)$ is the least $k \geq 0$ such that $w = u_1 \cdots u_k$ for reduced words $u_1, \dotsc, u_k$.
  \end{enumerate}
  Consequently, if $C$ is nonsingular, then every word $w \in (A \cup X)^\ast$ factors into its syllables and, in particular, if $w$ is not reduced, then some factor of length two of $w$ is not reduced.
\end{lemma}
\begin{proof}
  \textit{Ad \ref{lem:syllable-basics-1}.}
  If two distinct syllables shared a constant or a nonsingular variable, then, by locality, their union would be a reduced factor, contradicting maximality.

  \smallskip

  \textit{Ad \ref{lem:syllable-basics-2}.}
  Cutting $w$ at the starting positions of its syllables yields a factorization $w = u_1 \cdots u_k$ with $k = \syl(w)$, in which each $u_i$ is a factor of a syllable and, hence, reduced.

  Conversely, let $w = u_1 \cdots u_m$ with reduced words $u_1, \dotsc, u_m$.
  We may assume that $w$ contains a constant or a nonsingular variable, as $w$ is reduced otherwise.
  Then so does every syllable $v$, and we let $i_v$ be the position of the first such symbol in $v$.
  If $i_v < i_{v'}$ were two such positions within the same factor $u_j$, then, by locality, the union of $v$ with the reduced factor from $i_v$ to $i_{v'}$ would be reduced.
  By maximality, this union would be equal to $v$, contradicting the first item.
  Hence, distinct syllables $v$ have their positions $i_v$ in distinct factors $u_j$.
  This shows that $\syl(w) \leq m$.
\end{proof}

As we will now show, substitutions may merge syllables, but they never split them.
The number of syllables thus serves as a measure of progress for symbolic reduction.

\begin{lemma}\label{lem:syllables}
    Let $\tau \colon C \to C'$ be a substitution.
    Then, for every expression $w$ over $C$,
    \[
        \syl(\tau_\ast(w)) \leq \syl(w).
    \]
\end{lemma}
\begin{proof}
  Let $C' = (X', \mu')$ and $k = \syl(w)$.
  By \cref{lem:syllable-basics}~\ref{lem:syllable-basics-2}, we have $w = u_1 \cdots u_k$ with reduced words $u_i$ and, since $\mu = \mu' \circ \tau$, the words $\tau(u_i)$ are reduced over $C'$.
  Deleting a factor $\bar v v$ from such a product of $k$ reduced words again yields a product of $k$ reduced words.
  To see this, it suffices to consider the case of a single symbol $v \in A \cup X'$.
  If $\bar v$ and $v$ lie in the same reduced word, then $v$ is a singular variable, as $(\nu \circ \mu')(\bar v v) = 0$ otherwise.
  In this case, the deletion does not change the image under $\nu \circ \mu'$, so that the resulting word is still reduced.
  If $\bar v$ and $v$ lie in different reduced words, the deletion merely shortens these words.
  This shows that $\tau_\ast(w) = \red(\tau(w))$ is a product of $k$ reduced words, and \cref{lem:syllable-basics}~\ref{lem:syllable-basics-2} yields the claim.
\end{proof}

\subsection{An Automaton for Symbolic Reduction}\label{sec:automaton}

Symbolic reduction is best phrased in terms of an automaton that recognizes the substitutions achieving a complete symbolic reduction, or solving a given quadratic system of twisted equations.
We first describe an infinite such automaton $\mathcal{A}$, which carries no distinguished initial or final states, and then single out finite subautomata $\mathcal{A}_\lambda$ (together with a choice of initial and final states) that suffice for the specific task at hand.

\medskip

We define a \emph{state} of $\mathcal{A}$ as a pair $E = (C, w)$ comprising a context $C$ and a quadratic expression~$w$ over it, where the context is considered only up to isomorphism.
We also write $E = (X, \mu, w)$ in case $C = (X, \mu)$.
We impose three restrictions on states: the context $C$ is feasible and nonsingular, and no palindromic variable occurs in $w$, that is, $\abs{w}_x = 0$ for every palindromic $x \in X$.
Feasibility excludes states without terminal substitutions, and it is easily tested; see \cref{sec:expressions}.
The other two restrictions are for ease of exposition only and offending variables can simply be eliminated or split; see \cref{lem:splitting}.

For a state $E = (C, w)$, we write $\tsub(E) = \tsub(C)$, and we call $E$ \emph{terminal} if $C$ is the terminal context $C_\bot$.
Moreover, we call $E$ \emph{reduced} if $w$ is reduced over $C$; in particular, every terminal state is reduced.
Our goal will be to describe the reduced states, and in particular the terminal states, reachable in $\mathcal{A}$ from a given set of initial states.

The transitions of $\mathcal{A}$ are substitutions between the contexts, subject to the requirement that they transform the expressions accordingly.
In more detail, a \emph{transition} $E \xrightarrow{\tau} E'$ between states $E = (C, w)$ and $E' = (C', w')$ is a basic substitution $\tau \colon C \to C'$ with $\tau_\ast(w) = w'$.
Here, a substitution $\tau \colon (X, \mu) \to (X', \mu')$ is called \emph{basic} if $\abs{\tau(x)} \leq 3$ for all variables~$x \in X$.

Since contexts are considered only up to isomorphism, the above length bound ensures that every state of $\mathcal{A}$ has only finitely many outgoing transitions.
The automaton $\mathcal{A}$ itself, however, is infinite.
To pass to a finite automaton, we define the \emph{weight} of $E = (X, \mu, w)$ as
\[
    \lambda(E) = 4 \cdot \syl(w) + 2 \cdot \abs{X/G} + \abs{w};
\]
the reason for this choice will become apparent in the proof of \cref{thm:quadratic}.
For $\lambda \geq 0$, we then write $\mathcal{A}_\lambda$ for the restriction of $\mathcal{A}$ to those states $E$ with $\lambda(E) \leq \lambda$, together with the transitions between them.
Since both the number of variables and the length of the expression of a state are bounded in terms of its weight, the subautomaton $\mathcal{A}_\lambda$ of $\mathcal{A}$ is finite.
In fact, it has $2^{\mathcal{O}(\lambda \log \lambda)}$ states, with $2^{\mathcal{O}(\lambda \log \lambda)}$ outgoing transitions each, where the implicit constants depend only on $G$, $A$, and $N$.

\medskip

The raison d'être of the automata $\mathcal{A}_\lambda$ is that they allow us to perform symbolic reduction.
Each transition $E \xrightarrow{\tau} E'$ from $E = (C, w)$ to $E' = (C', w')$ can be thought of as revealing information about an a priori unknown terminal substitution $\sigma \in \tsub(E)$ which factors as $\sigma = \sigma' \circ \tau$ for some $\sigma' \in \tsub(E')$.
This is sound: since $w' = \tau_\ast(w) = \red(\tau(w))$ by definition, the word $\sigma(w) = \sigma'(\tau(w))$ reduces to $\sigma'(w')$.
The following theorem shows that symbolic reduction is also complete.

\begin{theorem}\label{thm:quadratic}
    Let $E = (X, \mu, w)$ be a state of $\mathcal{A}_\lambda$ and $\sigma \in \tsub(E)$.
    Then there exist a path
    \[
        E = E_0 \xrightarrow{\tau_1} E_1 \xrightarrow{\tau_2} \dotsb \xrightarrow{\tau_n} E_n = E'
    \]
    in $\mathcal{A}_\lambda$ to a reduced state $E' = (X', \mu', w')$ and $\sigma' \in \tsub(E')$ with $\sigma = \sigma' \circ \tau_n \circ \dotsb \circ \tau_2 \circ \tau_1$.
\end{theorem}

In the notation of \cref{thm:quadratic}, we have $\red(\sigma(w)) = \sigma'(w')$; in other words, symbolic reduction truly implements free reduction at the level of expressions.
Before we turn to the proof of \cref{thm:quadratic}, which is given in \cref{sec:completeness}, let us first explore some of its applications.

\medskip

To apply \cref{thm:quadratic} to a concrete problem, we need to choose initial states.
Typically, such a problem is given by a quadratic expression $w$ over a $G$-set of variables $X$ whose constraints are only partially prescribed, for instance by rational constraints on some of the variables.
We first add the missing details: we consider all contexts $C = (X, \mu)$ compatible with the prescription, that is, all admissible choices of $\mu(x) \in N$ for representatives $x$ of the orbits of $X$, and discard the infeasible ones.
For each of the remaining contexts $C$, we then split palindromic variables and eliminate singular ones as in the proof of \cref{lem:splitting}.
Accordingly, the \emph{initial states} of $(C, w)$ are all states $(C', \tau_\ast(w))$, where $\tau \colon C \to C'$ is a substitution of the form constructed in that proof; the initial states of the problem are those of all pairs $(C, w)$ obtained in the first step.
By \cref{lem:splitting,lem:syllables}, each of these initial states has weight at most $\lambda = 7 \cdot \abs{w} + 2 \cdot \abs{X/G}$.

By \cref{thm:quadratic,lem:splitting}, this choice of initial states is both sound and complete, and it yields the following parametric description of the values of quadratic expressions.

\begin{theorem}\label{thm:parametric}
  Let $w$ be a quadratic expression over $C$.
  Then one can compute reduced states $(C_1, w_1), \dotsc, (C_k, w_k)$ and substitutions $\tau_i \colon C \to C_i$ with $(\tau_i)_\ast(w) = w_i$ such that
  \[
    \{\, \sigma_\ast(w) : \sigma \in \tsub(C) \,\}
    =
    \bigcup_{i = 1}^{k} \{\, \sigma'(w_i) : \sigma' \in \tsub(C_i) \,\}
  \]
  or, more precisely, $(\sigma' \circ \tau_i)_\ast(w) = \sigma'(w_i)$ for all $\sigma' \in \tsub(C_i)$.
  Moreover, the states $(C_i, w_i)$ can be computed in time $2^{\mathcal{O}(\lambda \log \lambda)}$, where $\lambda = 7 \cdot \abs{w} + 2 \cdot \abs{X/G}$ for $C = (X, \mu)$, and the implicit constant depends only on $G$, $A$, and $N$; in particular, $k \in 2^{\mathcal{O}(\lambda \log \lambda)}$.
\end{theorem}
\begin{proof}
  Let $C = (X, \mu)$ and $\lambda = 7 \cdot \abs{w} + 2 \cdot \abs{X/G}$.
  Let $(C_1, w_1), \dotsc, (C_k, w_k)$ be the reduced states reachable in $\mathcal{A}_\lambda$ from the initial states of $(C, w)$, and let $\tau_i$ be the composition of the substitutions leading from $C$ to an initial state and then along a path to $(C_i, w_i)$; thus, $(\tau_i)_\ast(w) = w_i$.
  For $\sigma' \in \tsub(C_i)$, the word $\sigma'(w_i)$ is freely reduced since $w_i$ is reduced; hence, $(\sigma' \circ \tau_i)_\ast(w) = \sigma'_\ast(w_i) = \sigma'(w_i)$.
  Conversely, for every $\sigma \in \tsub(C)$, \cref{lem:splitting,thm:quadratic} yield a path from an initial state to some reduced state $(C_i, w_i)$ and $\sigma' \in \tsub(C_i)$ with $\sigma_\ast(w) = \sigma'(w_i)$.
  Since this value depends only on the reduced state reached, and not on the path, the claim follows.
  The time bound follows from those on the number of states and transitions of $\mathcal{A}_\lambda$.
\end{proof}

\begin{remark}\label{rem:parametric}
  \Cref{thm:parametric} describes the values of $w$ by finitely many parametric forms in which no cancellation occurs; in this sense, symbolic reduction eliminates free cancellation and reduces questions about the values of $w$ in $F(A)$ to questions about words in $A^\ast$.
  In terms of unification, it provides a finite complete set of parametric solutions of $w = y$ with respect to a new variable $y$.
  For the substitutions themselves, a finite description of this kind cannot be expected in general; consider, e.g., the equation $\bar x \bar y x y = 1$, whose solutions $\sigma$ are exactly those for which $\sigma(x)$ and $\sigma(y)$ are powers of a common element.
  Here, the paths of $\mathcal{A}_\lambda$ are needed, which leads to the \EDTOL\ descriptions of \cref{cor:quadratic-equations}.
\end{remark}

Using the automata $\mathcal{A}_\lambda$, we can also decide whether a given quadratic system of twisted equations has a solution and even effectively describe its solution set.
Here, a system
\[
  S = \{u_1 = v_1, \dotsc, u_n = v_n\}
\]
is called \emph{quadratic} if the word $u_1 v_1 \cdots u_n v_n$ satisfies the occurrence bounds of \cref{def:quadratic}.

\begin{corollary}\label{cor:quadratic-equations}
    On input of a quadratic system of twisted equations $S = \{u_1 = v_1, \dotsc, u_n = v_n\}$ with rational constraints in the free group $F(A)$, one can decide in \PSPACE\ whether $\sol(S) \neq \emptyset$.
    Moreover, $\sol(S)$ has an \EDTOL\ description computable by a \PSPACE-transducer.
\end{corollary}
\begin{proof}
    We encode $S$ as a single quadratic expression.
    To this end, we extend the alphabet $A$ by a new orbit $\#^G$ of separators on which $G$ acts freely, and we extend $\mu_\bot$ and the constraints of $S$ accordingly, such that variables still take values in $(A \setminus \#^G)^\ast$.
    Since no separator can cancel, a terminal substitution $\sigma$ solves $S$ if and only if $\sigma_\ast(w_S) = \#^{n-1}$, where $w_S = \red(\bar u_1 v_1 \# \cdots \# \bar u_n v_n)$.

    Let $\lambda$ be as above for $w_S$, and let $(C_\bot, \#^{n-1})$ be the single final state.
    We claim that $\sol(S)$ consists exactly of the compositions $\tau_m \circ \dotsb \circ \tau_1 \circ \tau$, where $\tau$ leads to an initial state of $w_S$ as above, and $\tau_1, \dotsc, \tau_m$ are the substitutions along a path in $\mathcal{A}_\lambda$ from this initial state to the final state.
    Soundness is clear.
    Conversely, for $\sigma \in \sol(S)$, \cref{thm:quadratic} yields a path to a reduced state $(C', w')$ and $\sigma' \in \tsub(C')$ with $\sigma'(w') = \#^{n-1}$; and, since $C'$ is nonsingular, this forces $w' = \#^{n-1}$.
    Since no variable occurs in $w' = \#^{n-1}$, the variables of~$C'$ can then be eliminated one orbit at a time by the basic substitutions defined by $x \mapsto \sigma'(x)$ if $\abs{\sigma'(x)} \leq 3$, and by $x \mapsto a_1 x a_2$ with $\sigma'(x) \in a_1 A^+ a_2$ otherwise.
    The first substitution eliminates $x$ and the second shortens the value of $x$, while neither of them increases the weight.

    By the above, $\sol(S) \neq \emptyset$ if and only if the final state is reachable in $\mathcal{A}_\lambda$ from some initial state.
    As $\lambda$ is linear in the input size, this can be decided in nondeterministic polynomial space and thus, by Savitch's theorem, in \PSPACE.
    Moreover, after fixing finitely many variable names, $\mathcal{A}_\lambda$ becomes a finite automaton over a monoid of endomorphisms of a free monoid accepting $\sol(S)$, and this automaton can clearly be output by a \PSPACE-transducer.
    The \EDTOL\ description of $\sol(S)$ now follows from the known conversion of such automata into \EDTOL\ systems; see~\cite{Asveld77, CiobanuDE16}.
\end{proof}

Note that, applied to the equations $x_1 \cdots x_n = u$ and $x_1 \cdots x_n = y$ from \cref{exm:palindrome-equation}, which are quadratic by \cref{exm:palindrome-quadratic}, \cref{cor:quadratic-equations} recovers \cref{cor:palindrome-length-pspace} by elementary means, that is, without resorting to the general results for equations in virtually free groups~\cite{DiekertE20}.

Since quadratic equations in free groups are a special case (see \cref{exm:quadratic-general}), deciding the solvability of quadratic twisted equations is \NP-hard by a result of Kharlampovich, Lys\"enok, Myasnikov, and Touikan~\cite{KharlampovichLMT10}.
They also show that quadratic equations in free groups without constraints can be solved in \NP; we do not know whether this upper bound extends to twisted equations or to equations with rational constraints.
\begin{remark}\label{rem:per-variable-constraints}
  In \cref{cor:quadratic-equations}, as throughout, the constraints of all variables are expressed via the single fixed $G$-homomorphism $\mu_\bot \colon A^\ast \to N$.
  If, instead, the rational constraints are part of the input and, moreover, given by homomorphisms to finite monoids (or by finite automata) \emph{independently} for each variable, then deciding the solvability of quadratic systems of twisted equations becomes \PSPACE-complete, provided that $A$ is sufficiently large.
  The upper bound follows as in the proof of \cref{cor:quadratic-equations}, whereas the lower bound follows by a straightforward reduction from the intersection problem for finite automata, which is \PSPACE-complete~\cite{Kozen77}; cf.~\cite{DiekertGH05} for the corresponding result on equations with rational constraints in free groups.
\end{remark}

\subsection{Completeness of Symbolic Reduction}\label{sec:completeness}

We now prove \cref{thm:quadratic}.
The key idea is to reveal~--~one basic substitution at a time~--~how the values of the variables in a nonreduced factor of length two may cancel, and then to carry out the corresponding free reduction symbolically.
Our definition of the weight ensures that this is possible without ever leaving the finite automaton $\mathcal{A}_\lambda$.

\begin{proof}[Proof of \cref{thm:quadratic}]
    Let $E = (X, \mu, w)$ be a state of $\mathcal{A}_\lambda$, and let $\sigma \in \tsub(E)$.
    If $E$ is a reduced state, then we can take the empty path from $E$ to itself and set $\sigma' = \sigma$.
    Thus, we may assume that~$E$ is not reduced.
    We claim that there then exist a transition $E \xrightarrow{\tau} E'$ in $\mathcal{A}_\lambda$, with $E' = (X', \mu', w')$, and a terminal substitution $\sigma' \in \tsub(E')$ such that $\sigma = \sigma' \circ \tau$ and $\abs{\sigma'(w')} < \abs{\sigma(w)}$ hold.
    Repeatedly following such a transition then yields the desired path to a reduced state.

    Since $w$ is not reduced and $C$ is nonsingular, $w$ contains a nonreduced factor of length two by \cref{lem:syllable-basics} and, since $w$ is freely reduced, at least one symbol occurring in such a factor is a variable.
    In the following, we distinguish several cases, based on the exact form of such a factor.
    In every case, we perform a substitution~$\tau$ such that $\tau(w)$ is not \emph{freely} reduced; hence, $\abs{\sigma'(w')} = \abs{\sigma'(\red(\tau(w)))} < \abs{\sigma'(\tau(w))} = \abs{\sigma(w)}$.

    \begin{enumerate}[label=(\arabic*)]
    \item The word $w$ contains a nonreduced factor $xa$ or $ax$ with $x \in X$ and $a \in A$.
    By symmetry, we may assume a factor of the form $xa$.
    As this factor is not reduced, we must have $\sigma(x) = x' \bar a$ for some $x' \in A^\ast$.
    Suppose first that $x' \neq 1$.
    We then apply the basic substitution $\tau$ defined by $x \mapsto x \bar a$.
    It leads to a new state $E' = (X', \mu', w')$ with $X' = X$ and a new $\sigma' \in \tsub(E')$ with $\sigma = \sigma' \circ \tau$, where $\sigma'(x) = x'$ (and $\sigma'(y) = \sigma(y)$ for all $y \in X \setminus x^G$).

    A routine computation shows that $\lambda(E') \leq \lambda(E)$.
    Indeed, $w' = \tau_\ast(w)$ satisfies \[
    \abs{w'} = \abs{\red(\tau(w))} \leq \abs{\tau(w)} - 2 \leq \abs{w},
    \]
    where the first inequality holds since $\tau(w)$ contains the factor~$\bar a a$, and the second since $w$ is quadratic.
    Moreover, $\abs{X'/G} = \abs{X/G}$ since $X' = X$, and $\syl(w') \leq \syl(w)$ by \cref{lem:syllables}.

    The case $x' = 1$ is handled by applying the substitution $\tau$ defined by $x \mapsto \bar a$, which eliminates the variable $x$; that is, $X' = X \setminus x^G$.
    A similar computation then shows that $\lambda(E') < \lambda(E)$.
    \end{enumerate}
    For the remaining cases, we consider nonreduced factors $xy$ with $x,y \in X$ and distinguish cases according to the possible relations between $x$ and $y$.
    In each case, we exhibit a substitution~$\tau$ and a terminal substitution~$\sigma'$ with $\sigma = \sigma' \circ \tau$ such that $\sigma'(\tau_\ast(xy))$ is freely reduced.
    By construction, this word equals $\red(\sigma(xy))$, which is thus the concatenation of a prefix of~$\sigma(x)$ and a suffix of~$\sigma(y)$.

    If this prefix and this suffix are both nonempty, then $\syl(w') < \syl(w)$ where $w' = \tau_\ast(w)$.
    Indeed, the factor $xy$ crosses two consecutive syllables of $w$, say $u = \tilde u x$ and $v = y \tilde v$.
    Since this prefix and this suffix are nonempty, $\sigma'(\tau_\ast(uv))$ is freely reduced as well, so that $\tau_\ast(uv)$ is reduced.
    Hence, as in the proof of \cref{lem:syllables}, $w'$ is a product of $\syl(w) - 1$ reduced words; that is, the syllables $u$ and $v$ have merged.

    A substitution $\tau$ that decreases the syllable count may increase the number of orbits of variables by at most one and the length of the expression by at most two; this is accounted for in the definition of the weight.
    If the syllable count does not decrease, that is, if $\sigma(x)$ or $\sigma(y)$ cancels entirely, then $\tau$ eliminates that variable and does not increase the length of the expression.
    For each of the cases below, upper bounds on the effect of $\tau$ on these three parameters are collected in \cref{tbl:substitutions}; verifying them is routine, and we leave it to the reader.

    \begin{table}[ht]
      \caption{Substitutions reducing a factor $xy$ with $x,y \in X$.}\label{tbl:substitutions}
      \begin{center}
      \renewcommand{\arraystretch}{1.2}
      \begin{tabular}{*{3}{cwc{5em}}ccc}
        \toprule
        &\multicolumn{3}{c}{Substitution} & \multicolumn{3}{c}{Upper Bounds} \\
        \cmidrule(r){2-4}\cmidrule(l){5-7}
             &$\tau_\ast(xy)$ & $\tau(x)$ & $\tau(y)$ & Syllables & Variables & Length \\
        \midrule

        \multirow{4}{*}{(2)}
            & $x y$ & $x \bar z$ & $z y$ & $-1$ & $+1$ & $+2$\\
            & $x$ & $x \bar z$ & $z$ & $0$ & $0$ & $0$\\
            & $y$ & $\bar z$ & $z y$ & $0$ & $0$ & $0$\\
            & $1$ & $\bar z$ & $z$ & $0$ & $-1$ & $-2$\\

        \midrule

        \multirow{2}{*}{(3)}
            & $x y$ & $x \bar z$ & $z y$ & $-1$ & $+1$ & $0$\\
            & $1$ & $\bar z$ & $z$ & $0$ & $0$ & $-2$\\

        \midrule

        \multirow{2}{*}{(4a)}
            & $z^{\bar \alpha}xy\bar z^\alpha$ & $z^{\bar \alpha}x \bar z$ & $z y \bar z^\alpha$ & $-1$ & $+1$ & $+2$ \\
            & $z^{\bar \alpha}\bar z^\alpha$ & $z^{\bar \alpha} \bar z$ & $z \bar z^\alpha$ & $-1$ & $0$ & $0$ \\

        \cmidrule(lr){2-7}

        \multirow{2}{*}{(4b)}

            & $z^{\bar \alpha} \bar z^\alpha$ & $z^{\bar \alpha} \bar q \bar z$ & $z q \bar z^\alpha$ & $-1$ & $+1$ & $0$ \\
            & $1$ & $\bar q$ & $q$ & $0$ & $0$ & $-2$ \\

        \bottomrule
      \end{tabular}
      \end{center}
    \end{table}

    \begin{enumerate}[label=(\arabic*)]
    \setcounter{enumi}{1}
    \item The word $w$ contains a nonreduced factor $xy$ with $x,y \in X$ and $x^G \neq y^G$.
    As this factor is not reduced, we must have $\sigma(x) = x' \bar z'$ and $\sigma(y) = z' y'$ for some $x', y' \in A^\ast$ and $z' \in A^+$.
    We choose $z'$ maximally, so that $x'y'$ is freely reduced.
    We then introduce a new variable~$z$ with stabilizer $G_z = G_{z'} \cap G_+$.
    Note that $G_x, G_y \subseteq G_z$.

    If $x' \neq 1$ and $y' \neq 1$, then we apply the substitution $\tau$ defined by $x \mapsto x\bar z$ and $y \mapsto zy$.
    This substitution leads to $E' = (X', \mu', w')$, where $X' = X \cup z^G$, and $\sigma' \in \tsub(E')$ with $\sigma = \sigma' \circ \tau$, where $\sigma'(x) = x'$, $\sigma'(y) = y'$, and $\sigma'(z) = z'$.
    If $x' = 1$, then we modify $\tau$ by requiring $x \mapsto \bar z$ and remove $x^G$ from $X'$.
    Likewise, if $y' = 1$, then we modify $\tau$ by requiring $y \mapsto z$ and remove $y^G$ from $X'$.

    \item The word $w$ contains a nonreduced factor $xy$ with $x,y \in X$ and $x^\alpha = y$ for some $\alpha \in G_-$.
    As above, let $\sigma(x) = x' \bar z'$ and $\sigma(y) = z'y'$ with $x',y' \in A^\ast$ and $z' \in A^+$ such that $x'y'$ is freely reduced.
    Let $z$ be a new variable with stabilizer $G_z = G_{z'} \cap G_+$.
    Note that $x^\alpha = y$ implies $x'^\alpha = y'$ as well as $z'^{\iota\alpha} = \bar z'^{\alpha} = z'$ and, hence, $\iota\alpha \in G_z$.

    If $x' \neq 1$ or, equivalently, $y' \neq 1$, then we apply the substitution $\tau$ defined by $x \mapsto x\bar z$, which satisfies $\tau(y) = zy$ by the previous consideration.
    As above, $\tau$ leads to a state $E' = (X', \mu', w')$ and $\sigma' \in \tsub(E')$ with $\sigma = \sigma' \circ \tau$, where $\sigma'(x) = x'$, $\sigma'(y) = y'$, and $\sigma'(z) = z'$.
    If $x' = y' = 1$, then we modify $\tau$ by requiring $x \mapsto \bar z$ and remove $x^G = y^G$ from $X'$.

    \item The word $w$ contains a nonreduced factor $xy$ with $x,y \in X$ and $x^\alpha = y$ for some $\alpha \in G_+$.
    \begin{enumerate}
    \item If $\abs{\red(\sigma(xy))} \geq \tfrac{1}{2}\abs{\sigma(xy)}$, meaning that at most half of $\sigma(xy)$ cancels, then we can write $\sigma(x) = u' x' \bar z'$ and $\sigma(y) = z' y' v'$ for some $x',y' \in A^\ast$ and $u', z', v' \in A^+$ such that, firstly, $\abs{u'} = \abs{z'} = \abs{v'}$ and, secondly, $u'x'y'v'$ is freely reduced.
    As $x^\alpha = y$ with $\alpha \in G_+$, we have $u' = z'^{\bar \alpha}$ and $\bar z'^{\alpha} = v'$.
    Finally, let $z$ be a new variable with $G_z = G_{z'} \cap G_+$.

    If $x' \neq 1$ or, equivalently, $y' \neq 1$, then we apply the substitution $\tau$ defined by $x \mapsto z^{\bar \alpha} x \bar z$, which satisfies $\tau(y) = z y \bar z^\alpha$.
    It leads to a state $E' = (X', \mu', w')$, where $X' = X \cup z^G$, and $\sigma' \in \tsub(E')$ with $\sigma = \sigma' \circ \tau$, where $\sigma'(x) = x'$, $\sigma'(y) = y'$, and $\sigma'(z) = z'$.
    If $x' = y' = 1$, then we modify $\tau$ by requiring $x \mapsto z^{\bar \alpha}\bar z$ and remove $x^G = y^G$ from $X'$.
    \item If $\abs{\red(\sigma(xy))} < \tfrac{1}{2}\abs{\sigma(xy)}$, meaning that more than half of $\sigma(xy)$ cancels, then we can write $\sigma(x) = u' \bar q' \bar z'$ and $\sigma(y) = z' q' v'$ for some $q' \in A^+$ and $u', z', v' \in A^\ast$ such that, firstly, $\abs{u'} = \abs{z'} = \abs{v'}$ and, secondly, $u'v'$ is freely reduced.
    Here, $u' = z'^{\bar \alpha}$ and $\bar z'^{\alpha} = v'$ as before.
    Moreover, we have $q'^{\iota\alpha} = \bar q'^{\alpha} = q'$; in particular, $q'$ is palindromic as $\iota\alpha \in G_-$.
    Let $z$ and $q$ be new variables with stabilizers $G_z = G_{z'} \cap G_+$ and $G_q = G_{q'}$, respectively.

    If $z' \neq 1$, then we apply the substitution $\tau$ defined by $x \mapsto z^{\bar \alpha}\bar q \bar z$, which satisfies $\tau(y) = zq\bar z^\alpha$.
    The substitution $\tau$ leads to a state $E' = (X', \mu', w')$, where $X' = (X \cup q^G \cup z^G) \setminus x^G$, and $\sigma' \in \tsub(E')$ with $\sigma = \sigma' \circ \tau$, where $\sigma'(q) = q'$ and $\sigma'(z) = z'$.
    Note that $E'$ is indeed a state, since the palindromic variable $q$ does not occur in $w'$.
    If $z' = 1$, then we modify $\tau$ by requiring $x \mapsto \bar q$ and remove $z^G$ from $X'$.
    \end{enumerate}
    \end{enumerate}
    In each case, $E'$ is indeed a state: $C'$ is feasible as $\sigma' \in \tsub(C')$, nonsingular as all values of $\sigma'$ are nonempty, and $w'$ is quadratic, as is easily checked.
    The case distinction above covers all possible nonreduced factors of length two, and a direct inspection of the upper bounds in \cref{tbl:substitutions} reveals that $\lambda(E') \leq \lambda(E)$ holds in each case.
\end{proof}

\section{Palindrome Length in Free Groups}\label{sec:palindrome-length}

Applied to palindrome length in free groups, as defined in \cref{sec:introduction}, \cref{thm:parametric} yields a parametric description of the freely reduced words of bounded palindrome length.

\begin{corollary}\label{cor:palindrome-length-forms}
  For every $n \geq 0$, one can compute reduced states $E_1, \dotsc, E_k$ with $E_i = (C_i, w_i)$ such that the $\sigma(w_i)$ with $\sigma \in \tsub(E_i)$ are precisely the words $w \in \red(A^\ast)$ with $\pl(w) \leq n$.
\end{corollary}
\begin{proof}
  By \cref{exm:palindrome-expression,exm:palindrome-equation}, the freely reduced words $w$ with $\pl(w) \leq n$ are exactly the values of the quadratic expression $x_1 \cdots x_n$ over the contexts $(X_n, \mu)$ from \cref{exm:palindrome-expression}.
  Applying \cref{thm:parametric} to each of the feasible ones, with $\lambda = 9n$, yields the list.
\end{proof}

\begin{remark}\label{rem:number-of-states}
The length $k$ of the list in \cref{cor:palindrome-length-forms} is at most $2^{\mathcal{O}(n \log n)}$.
On the other hand, our results below imply that, for every $n \geq 1$ and over a sufficiently large alphabet $A$, we must have $k \geq \frac{1}{n}\tbinom{2n - 2}{n - 1} \in \Theta(4^n \cdot n^{-3/2})$.
\end{remark}

To investigate the structure of freely reduced words of bounded palindrome length, we have implemented the procedure from the proof of \cref{thm:quadratic}.
For simplicity, and to obtain results independent of the alphabet $A$, our implementation does not work with rational constraints; instead, it tracks the partition into syllables directly.
For $n = 2, 3, 4, 5, 6$, it yields lists of $26$, $251$, $3054$, $42269$, and $629148$ reduced expressions, respectively.

These lists contain a great deal of redundancy.
Splitting a palindromic variable as in \cref{lem:splitting} distinguishes whether its value has even or odd length, and eliminating singular variables leads to expressions that are merely special cases of others.
After we unified all expressions into their most general forms, exactly $1$, $2$, $5$, $14$, and $42$ forms remained for $n = 2, 3, 4, 5, 6$, respectively.
Moreover, a clear structure emerged for these expressions, which we describe in this section.

In \cref{sec:patterns}, we introduce patterns and state our characterization (\cref{thm:palindrome-length}).
Its proof, given in \cref{sec:pattern-reduction}, is direct but follows the idea of symbolic reduction.
In \cref{sec:palindrome-language}, we finally turn this characterization into context-free grammars (\cref{lem:palindrome-length-grammar,pro:palindrome-length-language}) and prove that $\pl$ can be computed in \LOGCFL\ (\cref{thm:palindrome-length-compute}).

\subsection{Patterns}\label{sec:patterns}

Throughout this section, the group $G = \{1, \iota, \rho, \iota\rho\}$ is as in \cref{exm:palindrome-symmetries}, and $A$ is a fixed $G$-set on which $\rho$ acts trivially; we impose no restriction on the action of $\iota$.

The freely reduced words $w \in A^\ast$ of palindrome length $\pl(w) \leq n$ are best described in terms of (\emph{complete}) \emph{patterns}, meaning words over $\{\,\mathord{\langle}\,, \,\mathord{\rangle}\,, \palin\}$ generated by the context-free grammar with the single nonterminal $Q$ and the productions
\[
  Q \to \langle\,Q\,Q\,\rangle \mathbin{\vert} \palin.
\]

The \emph{rank} $\mathrm{rk}(u)$ of a pattern $u$ is the number of occurrences of $\palin$ in $u$.
As the grammar above is unambiguous, the complete patterns of rank $n \geq 1$ correspond to full binary trees with $n$ leaves; hence, they are counted by the Catalan number $C_{n-1} = \tfrac{1}{n}\tbinom{2n-2}{n-1}$.
\begin{example}
The two complete patterns of rank three are
\[
  \langle\langle \palin \palin \rangle \palin \rangle
  \quad\text{and}\quad
  \langle \palin \langle \palin \palin \rangle\rangle.
\]
\end{example}

Intuitively, each occurrence of $\palin$ in a pattern stands for a palindrome, and each matching pair of brackets for a twisted conjugation $y^{\alpha_k} v y$ of the enclosed factor~$v$ with twist $\alpha_k \coloneqq (\iota\rho)^k\iota$, where $k \geq 1$ is the rank of $v$; thus, $\alpha_k = \iota$ if $k$ is even, and $\alpha_k = \rho$ if $k$ is odd.
We make this precise by associating an expression $w_u$ in the sense of \cref{def:expression} with each pattern $u$.

\begin{definition}\label{def:pattern}
  For a pattern $u$, the expression $w_u$ is obtained from $u$ by replacing each occurrence of $\palin$ by a new \emph{palindrome variable} $x$, and the brackets ${\langle}$ and $\,{\rangle}$ of each matching pair by $y^{\alpha_k}$ and $y$, respectively, for a new \emph{bracket variable} $y$, where $k$ is the rank of the factor that the pair encloses.

  Here, each palindrome variable $x$ has stabilizer $G_x = \{1, \rho\}$ and each bracket variable $y$ has stabilizer $G_y = \{1\}$.
  We denote the union of their orbits by $X_u$.
\end{definition}

\begin{example}
With palindrome variables $x_1, x_2, x_3$ and bracket variables $y_1, y_2$, we obtain
\[
  w_{\langle\langle \palin \palin \rangle \palin \rangle} = y_1^\rho y_2^\iota x_1 x_2 y_2 x_3 y_1
  \quad\text{and}\quad
  w_{\langle \palin \langle \palin \palin \rangle\rangle} = y_1^\rho x_1 y_2^\iota x_2 x_3 y_2 y_1.
\]
\end{example}

Clearly, each $w_u$ is a quadratic expression.

\begin{definition}\label{def:instance}
  An \emph{instantiation} of a pattern $u$ is a $G$-equivariant map $\sigma$ assigning freely reduced words to the variables in $X_u$; the freely reduced word~$\sigma_\ast(w_u)$ is the \emph{instance} of $u$ under $\sigma$.
  We call $\sigma$ \emph{reduced} if $\sigma(w_u) = \sigma_\ast(w_u)$; in this case, $\sigma_\ast(w_u)$ is a \emph{reduced instance}.
\end{definition}

Instantiations of $u$ are exactly the terminal substitutions from the contexts $(X_u, \mu)$, namely with $\mu = \mu_\bot \circ \sigma$ (where, e.g., $\mu_\bot = \mu_A$); in particular, $\sigma$ is reduced if and only if $w_u$ is reduced over $(X_u, \mu_\bot \circ \sigma)$, and a variable is singular if and only if its value is empty.

Our main application, proved in \cref{sec:pattern-reduction}, characterizes palindrome length in free groups in terms of complete patterns.

\begin{theorem}\label{thm:palindrome-length}
  Let $n \geq 1$, and let $w \in A^\ast$ be a freely reduced word.
  Then $\pl(w) \leq n$ if and only if $w$ is a reduced instance of a complete pattern of rank~$n$.
\end{theorem}

Note that this characterization contains no redundancies, as distinct complete patterns describe different forms.
Indeed, let $u$ be a complete pattern of rank~$n$, and let $\sigma$ be the generic instantiation of $u$ mapping the variables of $w_u$ to new letters, using distinct orbits of letters for distinct orbits of variables.
Further, suppose that $\sigma(w_u) = \sigma'(w_{u'})$ for some reduced instantiation $\sigma'$ of a complete pattern $u'$ of rank~$n$.
Since $\sigma(w_u)$ has no palindrome factors of length greater than one, the palindrome variables of $u'$ take exactly the values of the palindrome variables of $u$, which are the letters whose orbits occur only once in $\sigma(w_u)$.
If a bracket variable of $u'$ took a value of length $m \geq 2$, then $u$ would contain $m$ pairs of brackets whose opening brackets are consecutive and whose closing brackets are consecutive, which is clearly impossible.
Hence, by counting, every bracket variable of $u'$ takes a single letter as its value, so that $u'$ can be read off from $\sigma'(w_{u'}) = \sigma(w_u)$.

A similar argument justifies the lower bound stated in \cref{rem:number-of-states}: the generic instance of each complete pattern $u$ of rank~$n$ is a value of some state $E_i = (C_i, w_i)$, and since $w_i$ contains at most $n$ constants, whereas every orbit of variables occurring in $w_i$ occurs twice and takes nonempty values, $u$ can again be read off from $w_i$.

\subsection{Reducing Patterns}\label{sec:pattern-reduction}

We prove \cref{thm:palindrome-length} directly, but following the idea of symbolic reduction: starting from a factorization into $n$ palindromes, we cancel one pair of letters at a time while maintaining an instantiation of a pattern of rank~$n$.
Since brackets are introduced only where letters cancel, the intermediate patterns need not be complete patterns; instead, they are \emph{partial patterns}, that is, words generated by the grammar with the single nonterminal $Q$ and the productions
\[
  Q \to \langle\,Q\,Q\,\rangle \mathbin{\vert} Q\,Q \mathbin{\vert} \palin.
\]

The definitions from \cref{sec:patterns} carry over verbatim to partial patterns.
A \emph{subpattern} of a partial pattern is a factor that is itself a partial pattern.

\begin{remark}\label{rem:pattern-completion}
  As is evident from their grammars, every partial pattern $u$ arises from a complete one, $\hat u$ say, of the same rank by deleting pairs of brackets.
  Mapping their bracket variables to the empty word shows that every reduced instance of $u$ is also one of $\hat u$.
\end{remark}

The following lemma justifies the interpretation of patterns from \cref{sec:patterns}.

\begin{lemma}\label{lem:pattern-instances}
  Let $n \geq 1$ and $w \in \red(A^\ast)$.
  Then the following are equivalent.
  \begin{enumerate}[label=(\arabic*)]
    \item\label{lem:pattern-instances-1} The word $w$ satisfies $\pl(w) \leq n$.
    \item\label{lem:pattern-instances-2} The word $w$ is an instance of the partial pattern $\palin^n$.
    \item\label{lem:pattern-instances-3} The word $w$ is an instance of some partial pattern of rank $n$.
  \end{enumerate}
\end{lemma}
\begin{proof}
  Clearly, \ref{lem:pattern-instances-2} implies \ref{lem:pattern-instances-3}.

  \smallskip

  \textit{Ad \ref{lem:pattern-instances-1} $\Rightarrow$ \ref{lem:pattern-instances-2}.}
  By assumption, and since $\red$ is $G$-equivariant, we have $w \equiv w_1 \cdots w_n$ for palindromes $w_1, \dotsc, w_n \in \red(A^\ast)$.
  Thus, $w$ is the instance of $\palin^n$ under $x_i \mapsto w_i$, mapping the variables of $w_{\palin^n} = x_1 \cdots x_n$ to these palindromes $w_1, \dotsc, w_n$.

  \smallskip

  \textit{Ad \ref{lem:pattern-instances-3} $\Rightarrow$ \ref{lem:pattern-instances-1}.}
  We show, by induction on $u$, that $\pi_A(\sigma(w_u))$ is a product of $\mathrm{rk}(u)$ palindromes for all instantiations $\sigma$ of a partial pattern $u$.
  For $u = \palin$ and for products $u = u_1 u_2$, this is obvious.
  For $u = \langle v \rangle$ with $k = \mathrm{rk}(u)$, it suffices to show that $h^{\alpha_k} g_1 \cdots g_k h$ is a product of $k$ palindromes for all $h \in F(A)$ and palindromes $g_1, \dotsc, g_k \in F(A)$.
  Since $\alpha_{i} \rho = (\iota\rho)^{i+1} = (\iota\rho)^{i-1}$ for all $i \geq 1$, which is $1$ for $i = 1$ and $\alpha_{i-1} \iota$ for $i \geq 2$, we have
  \[
    h^{\alpha_k} g_1 \cdots g_k h =
    \big(h^{\alpha_k} g_1 h^{\alpha_k \rho}\big) \cdots \big(h^{\alpha_1}g_k h^{\alpha_1 \rho}\big),
  \]
  where the $i$-th factor $h^{\alpha_{k-i+1}} g_i h^{\alpha_{k-i+1} \rho}$ is clearly a palindrome.
\end{proof}

By \cref{rem:pattern-completion,lem:pattern-instances}, it remains to show that every $w \in \red(A^\ast)$ with $\pl(w) \leq n$ is a reduced instance of some partial pattern of rank $n$.
We measure progress towards this goal by the \emph{length} of an instantiation $\sigma$ of a partial pattern $u$, which we define as $\abs{\sigma(w_u)}$.

The following lemma, which we need at the end of the proof of \cref{lem:pattern-reduction}, shows that adding a pair of brackets around a subpattern is harmless: either the result is again a partial pattern, or the new pair of brackets can be absorbed.

\begin{lemma}\label{lem:bracket-wrap}
  Let $u = u_1 \, v\, u_2$ be a partial pattern with a subpattern $v$.
  Then either $u' = u_1 \, \langle v \rangle \, u_2$ is a partial pattern or, for every instantiation $\sigma'$ of $u'$, there is an instantiation $\sigma$ of $u$ with
  \[
    \sigma_\ast(w_u) = \sigma'_\ast(w_{u'})
    \quad\text{and}\quad
    \abs{\sigma(w_u)} \leq \abs{\sigma'(w_{u'})}.\footnotemark
  \]
\end{lemma}

\begin{proof}
  \footnotetext{Here, expressions for and instantiations of $u'$ are defined as in \cref{sec:patterns}.}
  Note that $u' = u_1 \,\langle v \rangle\, u_2$ is a partial pattern unless $u$ has one of the following forms:
  \[
    u = u_1' \langle\, v \,\rangle u_2'
    \quad\text{or}\quad
    u = u_1 \,\langle v' \rangle\, u_2
    \quad\text{or}\quad
    u = u_1 \,\palin\, u_2.
  \]
  In each of these cases, the bracket or palindrome variable corresponding to the displayed symbols can absorb the value of the new bracket variable of $u'$.
\end{proof}

\begin{lemma}\label{lem:pattern-reduction}
  If $w \in A^\ast$ is the instance of a partial pattern $u$ under a nonreduced instantiation, then it is also the instance of a partial pattern of the same rank under a shorter instantiation.
\end{lemma}
\begin{proof}
  Let $\sigma$ be a nonreduced instantiation of $u$ with $\sigma_\ast(w_u) = w$.
  We first normalize the pattern~$u$ without changing the word $\sigma(w_u)$.
  If the bracket variable of some pair $\langle v \rangle$ in $u$ is singular, then we simply replace $\langle v \rangle$ by the partial pattern $v$.
  Repeating this, we may assume that all bracket variables are nonsingular.
  Moreover, occurrences of $\palin$ with a singular palindrome variable may be moved past adjacent factors of the form $\palin$ or $\langle v \rangle$, changing neither the word $\sigma(w_u)$ nor the rank of any pair of brackets in the process.
  We use this freedom of movement in the following argument.

  Since $\sigma(w_u)$ is not freely reduced, whereas the values of all variables are, $w_u$ contains occurrences $z_1$ and $z_2$ of nonsingular variables, in this order, such that $\sigma(z_1) \sigma(z_2)$ is not freely reduced and all symbols between them in $w_u$ are singular palindrome variables.
  If $z_2$ is a palindrome variable or the first occurrence $y^{\alpha_k}$ of a bracket variable $y$, then moving these singular palindrome variables past the factor $\palin$ or $\langle v \rangle$ of $u$ starting with $z_2$ makes $z_1$ and $z_2$ adjacent; the case in which $z_1$ is a palindrome variable or the second occurrence $y$ of a bracket variable is handled symmetrically.

  Hence, we may assume that $z_1$ and $z_2$ are adjacent, unless $z_1 = y^{\alpha_k}$ and $z_2 = y$ belong to a pair of brackets $\langle \palin^k \rangle$ in $u$.
  In the latter case, we obtain $u'$ by replacing $\langle \palin^k \rangle$ with $\palin^k$:
  \begin{itemize}
    \item If $k$ is even, then $\alpha_k = \iota$, so that $\sigma(y^{\alpha_k} y) \equiv 1$ and we can let $\sigma'$ be the restriction of $\sigma$.
    \item If $k$ is odd, then $\alpha_k = \rho$, so that $\sigma(y^{\alpha_k} y)$ is a palindrome and we can let $\sigma'$ map one of the (singular) palindrome variables of $\palin^k$ to the free reduction of this palindrome instead.
  \end{itemize}
  In both cases, $\sigma'$ is an instantiation of $u'$ with $\sigma'_\ast(w_{u'}) = \sigma_\ast(w_u)$ and $\abs{\sigma'(w_{u'})} < \abs{\sigma(w_u)}$.

  \medskip

  From now on, we assume that $z_1$ and $z_2$ are adjacent.
  We cancel the last letter of $\sigma(z_1)$ against the first letter of $\sigma(z_2)$ by removing these letters from the values of the corresponding variables.
  As a palindrome variable thereby also loses a letter at its other end, and a bracket variable at its other occurrence, we introduce a new pair of brackets whose bracket variable $y$ restores these two letters.
  If, however, the value of a palindrome variable involved is a single letter, we merely rearrange $u$ instead.

  \begin{table}[ht]
    \caption{Rewriting an adjacent pair $z_1 z_2$ of occurrences in the proof of \cref{lem:pattern-reduction}.}\label{tbl:bracket-cases}
    \begin{center}
    \renewcommand{\arraystretch}{1.2}
    \begin{tabular}{ccccccc}
      \toprule
       $u$ & $\sigma(z_1)$ & $\sigma(z_2)$ & $u'$ & $\sigma'(z_1)$ & $\sigma'(z_2)$ & $\sigma'(y)$ \\
      \midrule
      \multirow{4}{*}{$\palin\,\palin$}
        & $a x_1' a$ & $\bar a x_2' \bar a$ & $\langle \palin\,\palin \rangle$ & $x_1'$ & $x_2'$ & $\bar a$ \\
        \cmidrule(lr){2-7}
        & $a$ & $\bar a x_2' \bar a$ & \multirow{3}{*}{$\palin\,\palin$} & $x_2'$ & $\bar a$ & -- \\
      & $a x_1' a$ & $\bar a$ &  & $a$ & $x_1'$ & -- \\
      & $a$ & $\bar a$ &  & $1$ & $1$ & -- \\
      \midrule
      \multirow{2}{*}{$\palin \langle v_0 \rangle$}
        & $a x_1' a$ & $\bar a y_2'^{\beta_2}$ & $\langle \palin \langle v_0 \rangle \rangle$ & $x_1'$ & $y_2'^{\beta_2}$ & $\bar a^{\beta_2}$ \\
      & $a$ & $\bar a y_2'^{\beta_2}$ & $\langle v_0 \rangle\, \palin$ & $\bar a^{\beta_2}$ & $y_2'^{\beta_2}$ & -- \\
      \cmidrule(lr){1-7}
      \multirow{2}{*}{$\langle v_0 \rangle\, \palin$}
        & $y_1' a$ & $\bar a x_2' \bar a$ & $\langle \langle v_0 \rangle\, \palin \rangle$ & $y_1'$ & $x_2'$ & $\bar a$ \\
      & $y_1' a$ & $\bar a$ & $\palin \langle v_0 \rangle$ & $y_1'$ & $a^{\beta_1}$ & -- \\
      \midrule
      \multirow{2}{*}{$\langle \palin\, v_0 \rangle$}
        & $y_1'^{\beta_1} a$ & $\bar a x_2' \bar a$ & $\langle \palin \langle v_0 \rangle \rangle$ & $y_1'^{\beta_1}$ & $x_2'$ & $a^{\beta_1}$ \\
      & $y_1'^{\beta_1} a$ & $\bar a$ & $\langle v_0\, \palin \rangle$ & $y_1'^{\beta_1}$ & $a^{\beta_1}$ & -- \\
      \cmidrule(lr){1-7}
      \multirow{2}{*}{$\langle v_0\, \palin \rangle$}
        & $a x_1' a$ & $\bar a y_2'$ & $\langle \langle v_0 \rangle\, \palin \rangle$ & $x_1'$ & $y_2'$ & $a$ \\
      & $a$ & $\bar a y_2'$ & $\langle \palin\, v_0 \rangle$ & $\bar a^{\beta_2}$ & $y_2'$ & -- \\
      \midrule
      $\langle \langle v_1 \rangle v_2 \rangle$ & $y_1'^{\beta_1} a$ & $\bar a y_2'^{\beta_2}$ & $\langle \langle v_1 \rangle \langle v_2 \rangle \rangle$ & $y_1'^{\beta_1}$ & $y_2'^{\beta_2}$ & $a^{\beta_1}$ \\
      \cmidrule(lr){1-7}
      $\langle v_1 \langle v_2 \rangle \rangle$ & $y_1' a$ & $\bar a y_2'$ & $\langle \langle v_1 \rangle \langle v_2 \rangle \rangle$ & $y_1'$ & $y_2'$ & $a^{\beta_1}$ \\
      \midrule
      $\langle v_1 \rangle \langle v_2 \rangle$ & $y_1' a$ & $\bar a y_2'^{\beta_2}$ & $\langle \langle v_1 \rangle \langle v_2 \rangle \rangle$ & $y_1'$ & $y_2'^{\beta_2}$ & $\bar a^{\beta_2}$ \\
      \bottomrule
    \end{tabular}
    \end{center}
  \end{table}

  All cases are listed in \cref{tbl:bracket-cases}, where the variable of $z_i$ is denoted by $x_i$ if it is a palindrome variable, and by $y_i$ if it is a bracket variable, in which case $k_i$ denotes the rank of its pair of brackets and $\beta_i \coloneqq \alpha_{k_i}$; moreover, $a \in A$ and $x_i', y_i' \in A^\ast$ (the column $u$ shows only the factor of $u$ containing $z_1 z_2$, and $v_0$, $v_1$, $v_2$ denote the remaining contents of the respective brackets).
  In each case, the table specifies $u'$ together with an instantiation $\sigma'$ of $u'$, which agrees with $\sigma$ except on the variables of $z_1$ and $z_2$ and, if present, on the new bracket variable $y$.
  It is routine to verify that
  \[
    \sigma'_\ast(w_{u'}) = \sigma_\ast(w_u)
    \quad\text{and}\quad
    \abs{\sigma'(w_{u'})} < \abs{\sigma(w_u)}
  \]
  hold in each case; we leave the details to the reader.
  Finally, if $u'$ is a partial pattern, we are done; in particular, this is the case if no new brackets are introduced.
  Otherwise, \cref{lem:bracket-wrap} yields an instantiation $\sigma''$ of the original partial pattern $u$ with 
  \[
    \sigma''_\ast(w_u) = \sigma_\ast(w_u)
    \quad\text{and}\quad
    \abs{\sigma''(w_u)} \leq \abs{\sigma'(w_{u'})} < \abs{\sigma(w_u)}.\qedhere
  \]
\end{proof}

\begin{proof}[Proof of \cref{thm:palindrome-length}]
  If $w$ is a reduced instance of a complete pattern of rank $n$, then $\pl(w) \leq n$ by \cref{lem:pattern-instances}.
  Conversely, suppose that $\pl(w) \leq n$.
  By \cref{lem:pattern-instances}, $w$ is an instance of a partial pattern of rank $n$; we choose such a pattern $u$ and an instantiation $\sigma$ of $u$ of minimal length with $\sigma_\ast(w_u) = w$.
  By \cref{lem:pattern-reduction}, $\sigma$ is reduced, so that $w = \sigma(w_u)$ is a reduced instance of $u$ and, by \cref{rem:pattern-completion}, of a complete pattern of rank~$n$.
\end{proof}

\subsection{A Formal Language Description}\label{sec:palindrome-language}

For each $n$, the characterization from \cref{thm:palindrome-length} can be turned into an explicit context-free grammar.
For simplicity, we first construct a grammar generating all freely reduced words~$w$ with $\pl(w) \leq n$ together with some nonreduced words satisfying the same bound.
When applied to these grammars, standard constructions then yield grammars for exactly the freely reduced words, or all words, of palindrome length at most~$n$ (\cref{pro:palindrome-length-language}).

For these grammars, we fix a countably infinite set of nonterminal symbols $\{Q_1, Q_2, \dotsc\}$.
The basic productions between these are given by
\begin{equation}\label{eqn:palindrome-grammar-1}
  Q_k \to Q_i Q_j, \quad
  Q_k \to a^{\alpha_k} Q_k \, a, \quad
  Q_1 \to aa \mathbin{\vert} a,
\end{equation}
where $k = i + j$, $a \in A$, and $\alpha_k = (\iota \rho)^k \iota$ is as in \cref{sec:patterns}.
Here, the second type produces brackets for $k \geq 2$; for $k = 1$, where $a^{\alpha_1} = a$, it produces, together with the last type, any nonempty palindrome from $Q_1$.
To account for singular values, we also include
\begin{equation}\label{eqn:palindrome-grammar-2}
  Q_k \to Q_{k-1} \mathbin{\vert} \cdots \mathbin{\vert} Q_2 \mathbin{\vert} Q_1.
\end{equation}

For each $n \geq 0$, we write $\mathcal{P}_n$ for the context-free grammar with nonterminals $Q_1, \dotsc, Q_n$ and a dedicated start nonterminal $S_n$ whose productions are given by \cref{eqn:palindrome-grammar-1,eqn:palindrome-grammar-2} together with the productions $S_n \to Q_n \mathbin{\vert} 1$ or, in case $n = 0$, with $S_0 \to 1$.

The following properties of $\mathcal{P}_n$ are now immediate from \cref{thm:palindrome-length,lem:pattern-instances}.

\begin{lemma}\label{lem:palindrome-length-grammar}
  Let $L(\mathcal{P}_n) \subseteq A^\ast$ denote the language generated by the grammar $\mathcal{P}_n$.
  \begin{enumerate}[label=(\arabic*)]
    \item If $w \in L(\mathcal{P}_n)$, then $\pl(w) \leq n$.
    \item If $w \in \red(A^\ast)$ and $\pl(w) \leq n$, then $w \in L(\mathcal{P}_n)$.
  \end{enumerate}
\end{lemma}

Using the well-known facts that context-free languages are closed under intersection~with regular ones (see, e.g., \cite{HopcroftU79})~--~here, $\red(A^\ast)$~--~and under taking ancestors with respect to finite monadic rewrite systems~\cite[Theorem~2.2]{BookJW82}~--~here, free reduction~--~we obtain the following.

\begin{proposition}\label{pro:palindrome-length-language}
  For every $n \geq 0$, the following languages are context-free.
  \begin{enumerate}[label=(\arabic*)]
      \item\label{pro:palindrome-length-language-1} The language $RP_n \subseteq A^\ast$ consisting of all freely reduced words $w$ with $\pl(w) \leq n$.
      \item\label{pro:palindrome-length-language-2} The language $AP_n \subseteq A^\ast$ consisting of all words $w$ with $\pl(w) \leq n$.
  \end{enumerate}
\end{proposition}
\begin{proof}
\textit{Ad \ref{pro:palindrome-length-language-1}.}
By \cref{lem:palindrome-length-grammar}, the language $RP_n$ equals $\red(A^\ast) \cap L(\mathcal{P}_n)$;
hence, as an intersection with a regular language, $RP_n$ is also context-free.

\smallskip

\textit{Ad \ref{pro:palindrome-length-language-2}.}
Let $\mathfrak{R}$ be the finite monadic rewrite system on $A^\ast$ given by the rules $\bar a a \to 1$ for all $a \in A$.
Clearly, the language $AP_n$ coincides with the language
\[
    \{
        w \in A^\ast :
        w \to^\ast_\mathfrak{R} w'
        \text{ for some }
        w' \in RP_n
    \}
\]
of all ancestors of words in $RP_n$ with respect to $\mathfrak{R}$.
Since context-free languages are closed under taking ancestors with respect to a finite monadic rewrite system~\cite[Theorem~2.2]{BookJW82}, we deduce that this language is indeed context-free.
\end{proof}

In view of the uniformity of the description of the context-free grammars $\mathcal{P}_n$, a straightforward adaptation of the algorithm of Cocke, Younger, and Kasami~\cite{CockeS70, Kasami65, Younger67} computes the palindrome length of a given word in cubic time.

\begin{corollary}\label{cor:palindrome-length-cubic}
  On input of a word $w \in A^\ast$, one can compute $\pl(w)$ in time $\mathcal{O}(\abs{w}^3)$.
\end{corollary}
\begin{proof}
  As free reduction can be computed in linear time, we may assume that $w$ is freely reduced.
  Since $\mathcal{P}_n$ contains the productions $Q_k \to Q_{k-1}$, the languages generated from $Q_1, Q_2, \dotsc$ form an ascending chain.
  Hence, it suffices to compute, for every factor $u$ of $w$, the least $k$ such that $u$ is generated from $Q_k$.
  By \cref{lem:palindrome-length-grammar}, this number equals $\pl(w)$ for $u = w$.

  Apart from the base cases, i.e., for $Q_1 \to aa \mathbin{\vert} a$, this number is determined by the corresponding numbers for the proper factors of $u$.
  Indeed, for $Q_k \to Q_i Q_j$, we simply minimize over the $\mathcal{O}(\abs{u})$ factorizations of $u$, and for $Q_k \to a^{\alpha_k} Q_k a$, it suffices to consider the least admissible $k$ of each parity.
  Computing these numbers for all factors, of which there are $\mathcal{O}(\abs{w}^2)$, in order of increasing length thus takes time $\mathcal{O}(\abs{w}^3)$.
\end{proof}

Beyond cubic time, the uniform grammars~$\mathcal{P}_n$ also yield a bound in terms of parallel complexity, namely that $\pl$ is computable\footnote{Meaning that, on input of $w \in A^\ast$ and $n \in \mathbb{N}$, one can decide whether $\pl(w) = n$.} in \LOGCFL.
Recall that \LOGCFL, introduced by Sudborough~\cite{Sudborough78}, is the class of all decision problems that are logspace-reducible to the membership problem of a context-free language.
It is known that $\LOGCFL \subseteq \NC^2$~\cite{Ruzzo80} and, due to the former inclusion, that $\LOGCFL \subsetneq \PSPACE$.
Although the target language in the above definition is fixed, $\LOGCFL$ crucially also contains the uniform membership problem for $1$-free\footnote{Without this restriction, the uniform membership problem is $\PTIME$-complete~\cite{JonesL76}.} context-free grammars~\cite{Goldschlager81, Ruzzo80}.

\begin{theorem}\label{thm:palindrome-length-compute}
    The function $\pl$ is computable in \LOGCFL.
\end{theorem}

\begin{proof}
    We first consider the problem of deciding, on input of a word $w \in A^\ast$ and a number~$n$, whether $\pl(w) \leq n$.
    Since clearly $\pl(w) \leq \abs{w}$, we may assume that $n \leq \abs{w}$.
    In particular, the number $n$ is effectively given in unary encoding, so a logspace-transducer can compute the context-free grammar~$\mathcal{P}_n$.
    Free reduction can also be computed by a logspace-transducer~\cite{ElderEO13}.
    Hence, we may reduce to the question of whether $\red(w) \in L(\mathcal{P}_n)$.
    By \cref{lem:palindrome-length-grammar}, we indeed have $\red(w) \in L(\mathcal{P}_n)$ if and only if $\pl(w) \leq n$.
    As the target of this reduction~--~the uniform membership problem for $1$-free context-free grammars~--~is in \LOGCFL, so is our problem.

    Finally, note that $\pl(w) = n$ if and only if $\pl(w) \leq n$ and \emph{not} $\pl(w) \leq n-1$.
    By the above, the first condition can be decided in \LOGCFL; since \LOGCFL\ is closed under complement~\cite{BorodinCDRT89}, so can the second.
    Therefore, we can decide in \LOGCFL\ whether $\pl(w) = n$.
\end{proof}

\section{Conclusion}\label{sec:conclusion}

In this article, we have investigated twisted equations in free groups, in which a finite group of symmetries acts by automorphisms and anti-automorphisms, and the variables may be subject to rational constraints.
For quadratic twisted equations, we have extended the algorithm of Comerford and Edmunds~\cite{ComerfordE81}: basic substitutions carry out free reduction symbolically, while a simple weight function confines the search to finite automata.
Besides deciding solvability in \PSPACE\ and yielding \EDTOL\ descriptions of solution sets, this algorithm computes, for every quadratic twisted expression, finitely many parametric forms describing its values without any cancellation.
This applies, e.g., to products of squares or commutators, for which Wicks-type forms are classical~\cite{Wicks62, Culler81}, and to products of twisted conjugates or palindromes with respect to arbitrary antimorphic involutions.

The most pressing open question concerns the complexity of solving quadratic twisted equations: deciding solvability is \NP-hard~\cite{KharlampovichLMT10}, but we do not know whether it is in \NP.
For quadratic equations without twists and constraints, membership in \NP\ is known in free groups~\cite{KharlampovichLMT10} and, more generally, in torsion-free hyperbolic groups~\cite{KharlampovichMTV17}.
It is open whether these bounds extend to twisted equations or to rational constraints (see \cref{rem:per-variable-constraints}).
It would also be interesting to extend symbolic reduction to quadratic (twisted) equations in other structures, in particular, hyperbolic groups and right-angled Artin groups.

As an application, we have studied palindrome length in free groups.
Guided by the parametric forms computed by our algorithm, we found, and then proved directly, a characterization of the freely reduced words of palindrome length at most~$n$, thereby confirming the Catalan structure conjectured by the second author~\cite{Frid25}.
We have also shown that, for every~$n$, these words form a context-free language, and that $\pl$ is computable in cubic time and in \LOGCFL.

The precise complexity of computing~$\pl$, however, remains open.
For words, palindrome length can be computed in linear time~\cite{BorozdinKRS17}, whereas our algorithm for free groups takes cubic time.
Can this gap be closed, or is there a matching conditional lower bound?
Similarly, is computing~$\pl$ hard for \LOGCFL, or does it lie in a smaller class such as \NL, as for words?

Finally, we hope that our article also serves as a gentle introduction to the theory of equations in free groups, in particular for readers from combinatorics on words.

\section*{Acknowledgements}

The authors thank Aix Marseille Université for financial support and the Centre international de rencontres mathématiques (CIRM) for its hospitality.

\bibliographystyle{abbrvurl}
\bibliography{main}

\end{document}